\documentclass
{article}
\usepackage{amsmath,amssymb,float}
\usepackage{amsfonts,amsthm}
\usepackage{hyperref}

\usepackage{xypic}

\usepackage{vmargin}

         \setmarginsrb{15mm}{5mm}{15mm}{15mm}%
         {10mm}{10mm}{10mm}{10mm}

\title{Metric vs topological receptive entropy of semigroup actions}
\author{Andrzej Bi\'s\footnote{University of Lodz - Poland, andrzej.bis@wmii.uni.lodz.pl, ORCID iD:0000-0002-3920-6044} \and Dikran Dikranjan\footnote{University of Udine - Italy, dikran.dikranjan@uniud.it, ORCID iD:0000-0002-1159-9958} \and Anna Giordano Bruno\footnote{corresponding author, University of Udine - Italy, anna.giordanobruno@uniud.it, ORCID iD:0000-0003-3431-2240} \and Luchezar Stoyanov\footnote{University of Western Australia - Australia, luchezar.stoyanov@uwa.edu.au, ORCID iD:0000-0002-4637-9277}}

\date{}

\makeindex

\newlength{\bibitemsep}
\newlength{\bibparskip}
\let\oldthebibliography\thebibliography
\renewcommand\thebibliography[1]{%
  \oldthebibliography{#1}%
  \setlength{\parskip}{\bibitemsep}%
  \setlength{\itemsep}{\bibparskip}%
}

\begin{document}


\maketitle

\def\R{{\mathbb R}}
\def\T{{\mathbb T}}
\def\S{{\mathbb S}}
\def\Z{{\mathbb Z}}
\def\C{{\mathbb C}}
\def\N{{\mathbb N}}

\def\pa{\underline{\mathfrak{par}}}
\def\paa{\mathrm{par}}
\def\cov{\mathfrak{cov}}
\def\covv{\mathrm{cov}}

\def\aa{{\mathcal A}}
\def\bb{{\mathcal B}}
\def\cc{{\mathcal C}}
\def\dd{{\mathcal D}}
\def\mm{{\mathcal M}}
\def\nn{{\mathcal N}}
\def\ee{{\mathcal E}}
\def\kk{{\mathcal K}}

\def\g{{\bf g}}
\def\h{{\bf h}}
\def\k{{\bf k}}
\def\kbar{\overline{k}}
\def\n{{\bf n}}
\def\r{{\bf r}}
\def\s{{\bf s}}
\def\ep{\epsilon}
\def\tr{\tilde{r}}
\def\ts{\tilde{s}}

\def\m{{\sf m}}
\def\om{\overline{\m}}

\def\Glc{{\rm Gl}(k,\C)}
\def\glc{{\rm gl}(k,\C)}
\def\sl{{\rm sl}(2,\R)}
\def\Gl{{\rm Gl}(2,\R)}
\def\SO{{\rm SO}(n)}
\def\e{\emptyset}
\def\dist{\mbox{\rm dist}}
\def\diam{\mbox{\rm diam}}
\def\con{\mbox{\rm const}}

\def\th{\tilde{h}}
\def\id{\mathop{\rm id}\nolimits}

\def\bs{\bigskip}
\def\ms{\medskip}
\def\endofproof{$\Box$\bs
}

\def\hh{\widehat{h}}
\def\tN{\widetilde{N}}
\def\Int{\mbox{\rm Int}}

\def\sp{$\spadesuit$}

\newtheorem{theorem}{Theorem}[section]
\newtheorem{fact}[theorem]{Fact}
\newtheorem{corollary}[theorem]{Corollary}
\newtheorem{lemma}[theorem]{Lemma}
\newtheorem{proposition}[theorem]{Proposition}
\newtheorem{conjecture}[theorem]{Conjecture}

\newtheorem{claim}
{Claim}

\numberwithin{equation}{section}

\newtheorem{problem}[theorem]{Problem}
   \newtheorem{construction}{Construction}

\theoremstyle{definition}
\newtheorem{definition}[theorem]{Definition}
\newtheorem{remark}[theorem]{Remark}
\newtheorem{remarks}[theorem]{Remarks}
\newtheorem{question}[theorem]{Question}
\newtheorem{example}[theorem]{Example}

\renewcommand{\theequation}{\arabic{section}.\arabic{equation}}

\abstract{We study the receptive metric entropy for semigroup actions on probability spaces, inspired by a similar notion of topological entropy introduced by Hofmann and Stoyanov \cite{HS}. We analyze its basic properties and its relation with the classical metric entropy. 
In the case of semigroup actions on compact metric spaces we compare the receptive metric entropy with the receptive topological entropy looking for a Variational Principle. With this aim we propose several characterizations of the receptive topological entropy. Finally we introduce a receptive local metric entropy inspired by a notion by Bowen generalized in the classical setting of amenable group actions by Zheng and Chen \cite{ZC}, and we prove partial versions of the Brin-Katok Formula and the local Variational Principle.}



\medskip
\noindent Data sharing not applicable to this article as no datasets were generated or analysed during the current study.

\section{Introduction}

%
 
In \cite{HS} a general notion of \emph{receptive topological entropy} (the term ``receptive'' was coined later on) was introduced and studied for a uniformly continous action $T\colon G\times X \to X$ of a locally compact semigroup $G$ on a metric space $(X,d)$. The concept defined there depends on a system $\Gamma = (N_0, N_{1},\ldots,N_{n},\ldots)$ of compact subsets of the acting semigroup $G$ satisfying $N_{k} N_{m} \subseteq N_{k+m}$ for every $k,m\in\Z_+$; such $\Gamma$ is called a \emph{regular system} in $G$. The receptive topological entropy of $T$ with respect to $\Gamma$ is defined by 
\begin{equation}\label{1.1}
\th (T,\Gamma) = \sup\left\{ \lim_{\ep \to 0}{\limsup_{n \rightarrow \infty}} \frac{1}{n} \log s_{n}(\ep , K, T) \colon K\subseteq X,\ K\ \text{compact}\right\},
\end{equation}
where $s_n(\ep,K,T)$ is the maximum size of an $(n,\ep)$-separated set of $K$ (i.e., a set $E \subseteq K$ such that for any distinct $x,y \in E$ there exists $g \in N_n$ with $d(gx,gy) > \ep$). 

In case $G=\Z_+$, the action $T$ is generated by the uniformly continuous selfmap $f=T(1,-)\colon X\to X$, and taking $N_n=[0,n]$ for every $n\in\Z_+$, the receptive topological entropy $\th(T,\Gamma)$ coincides with the classical Bowen-Dinaburg's topological entropy $h(f)$.
As in that situation, in \eqref{1.1} one can use the so called $(n,\ep)$-spanning sets and $r_n(\ep,K,T)$ (see Section~\ref{recsec}).

A topological entropy similar to \eqref{1.1} was studied, albeit in a different context, by Ghys, Langevin and Walczak \cite{GLW}, and later by Bi\'s \cite{Bis1}.

In the present paper we study a natural similar definition of a \emph{receptive metric entropy} $\th_\mu(T,\Gamma)$ of a measure-preserving action $T\colon G\times X\to X$ of a discrete semigroup $G$ on a probability space $(X,\mu)$ with respect to a regular system $(N_n)_{n\in\N}$ in $G$. 
Also in this case, for $G=\Z_+$, $f=T(1,-)$ and $\Gamma=([0,n])_{n\in\N_+}$, the receptive metric entropy $\th_\mu(T,\Gamma)$ coincides with the classical metric entropy $h_\mu(f)$ by Kolmogorov and Sinai.

\smallskip
The dependence of the receptive entropies on the regular system $\Gamma$ is present starting from the classical case of actions of ${\Z}_{+}$. Indeed the regular system  $\Gamma =([0,n])_{n\in\Z_+}$ is accepted as natural and is very rarely mentioned explicitly, however, in general $\th(T,\Gamma)$ varies with  $\Gamma$. 
Kushnirenko \cite{Ku} pointed this out for the case of the metric entropy of measure-preserving selfmaps $f$. He defined the concept of a metric $A$-entropy of $f$, where $A=(t_{n})_{n\in\Z_+}$ is a sequence of positive integers.  In particular, he proved that for $A=(2^{n})_{n\in\Z_+}$, the metric $A$-entropy is different from the classical metric entropy, which corresponds to the choice $A=(n)_{n\in\Z_+}$. 
A similar topological concept, called \emph{topological sequence entropy}, was introduced and studied several years later by Goodman \cite{Go}, who also investigated the relationship between his concept and the measure-theoretic analogue defined by Kushnirenko. After that there has been a significant activity in studying both kinds of sequence entropies  -- see for example \cite{De,Ca,FrS,Le}. 

On the other hand, there exist well-known definitions of metric entropy $h_\mu$ and topological entropy $h$ for amenable group and semigroup actions (see for example \cite{Ru1,Ru2,Con,Oll,Mi1,E,El} for these notions and their main properties) and in this case the value of the entropy does not depend on the choice of the F\o lner sequence used in the definition. 
However, with these classical definitions, both the metric and the topological entropy are frequently zero (see Section~\ref{classicalhm} and Section~\ref{classicaltop} below). For example this is the case for smooth actions $T$ of ${\Z}_{+}^{k}$, with $k > 1$, on manifolds: as mentioned in \cite{HS}, while $h(T)=0$, the receptive topological entropy $\th(T)$ is not trivially zero. 
Other examples are given below where $\th_\mu(T)$ and $\th(T)$ are non-zero, while the classical entropies $h_\mu(T)$ and $h(T)$ are both zero. Conversely, as one can see from the definitions, when $h_\mu(T)$ or $h(T)$ is non-zero our definitions produce $\infty$.

\medskip
In this paper first of all, in Section~\ref{recmetric} we study the basic properties of the receptive metric entropy $\th_\mu(T,\Gamma)$ of a measure-preserving action $T\colon G\times X\to X$ on a probability space $(X,\mu)$, where $\Gamma$ is a regular system in $G$. 
We compare $\th_\mu(T,\Gamma)$ with the classical metric entropy $h_\mu(T)$ in case $\Gamma$ is also a F\o lner sequence (see Section~\ref{classicalhm}), and we point out that the receptive metric entropy does not vanish in many of the cases when the classical one does (see Proposition~\ref{Proposition2.4} and Example~\ref{Example2.5}). Further properties concerning restriction actions and products of actions are discussed in Section~\ref{further}.

The receptive topological entropy $\th(T,\Gamma)$ is recalled in Section~\ref{rectop}, together with some results from \cite{HS}. Also in this case we recall that several natural actions where the receptive topological entropy does not vanish while the classical one does
 (see \cite{HS} and also \cite{Sch}).

\smallskip
In Section~\ref{equiv} we propose several different versions of the receptive topological entropy. We start considering a version using open covers for continuous actions on compact topological spaces, following the classical idea of Adler, Konhein and McAndrew \cite{AKMc}.

We generalize a well-known notion by Bowen \cite{Bow2} defining a receptive-like entropy $b(T,\Gamma)$ for a continuous action $T\colon G\times X\to X$ of a semigroup $G$ on a compact metric space $X$, where $\Gamma$ is a regular system in $G$.
Then $b(T,\Gamma)\leq\th(T,\Gamma)$, and the equality $b(T,\Gamma)=\th(T,\Gamma)$ holds when $G$ is commutative and finitely generated and $\Gamma$ is standard (i.e., $N_1$ generates $G$ and $N_n=N_1^n$ for every $n\in\Z_+$). This is proved in Theorem~\ref{Theorem5.3}, which very roughly follows the line of \cite{Bow2}, but its proof is more subtle and requires much more effort and computations.

It turns out that this receptive entropy $b(T,\Gamma)$ coincides with another entropy-like quantity  $c(T,\Gamma)$ defined by using some ideas of Pesin \cite{P} (see Theorem~\ref{Theorem5.5}).

\smallskip
In Section~\ref{VPsec}, we discuss the relation of the receptive topological entropy with the receptive metric entropy, looking for some Variational Principle. 
Following \cite[Chapter 6]{Wa},  for a compact metric space $X$ we consider Borel probability measures on $X$
having as $\sigma$-algebra of measurable sets precisely the $\sigma$-algebra of Borel sets; any such measure is necessarily inner and outer regular \cite[Theorem 6.1]{Wa}.
The classical {Variational Principle} due to Goodwyn \cite{G} and Goodman~\cite{Go1} states that, in case $G=\Z_+$ and $f=T(1,-)\colon X\to X$ is a continuous selfmap of the compact metric space $X$, then $$h(f)=\sup\{h_\mu(f)\colon \mu\ \text{$f$-invariant Borel probability measure}\}.$$
It is well-known that the Variational Principle holds also for continuous actions of amenable groups or of countable cancellative semigroups on compact metric spaces for the classical topological and metric entropy defined by means of a F\o lner sequence (see \cite{Oll,OllP} and \cite{Mi1,ST} respectively).

We conjecture that the Variational Principle holds for the receptive topological entropy and the receptive metric entropy (see Conjecture~\ref{MainConj}): if $G$ is a semigroup, $\Gamma=(N_n)_{n\in\Z_+}$ a regular system in $G$ and $T\colon G\times X\to X$ a continuous action on a compact metric space $X$, denote by $M(X,T)$ the family of all $T$-invariant Borel probability measure on $X$ and assume that $M(X,T)$ is non-empty; then 
\begin{equation}\label{VPintro}
\th(T,\Gamma) = \sup \{\th_\mu(T,\Gamma)\colon \mu\in M(X,T)\}.
\end{equation}
This remains an open problem, nevertheless we prove some partial results.
We give also a brief comment on the algebraic entropy and its connection to the topological entropy both in the classical and in the ``receptive'' cases.

\smallskip
To further clarify the relationships between the concepts considered in this paper, for a semigroup action $T\colon G\times X\to X$ on a compact metric space $X$ and $\Gamma$ a regular system in $G$, in Section~\ref{local}  we introduce a receptive local metric entropy $\underline h_\mu^{loc}(T,\Gamma)$, where $\mu$ is a Borel probability measure on $X$; this is defined by
$$\underline h_\mu^{loc}(T,\Gamma) = \int _X h_\mu^{loc}(x) \ d\mu,$$
where $h_\mu^{loc}(x)$  is the receptive local metric entropy at $x \in X$. 
The latter entropy was defined in \cite{Bis3} in the case when the acting semigroup is a finitely generated group and the regular system is standard, while an analogous local metric entropy was defined in \cite{ZC} for actions of amenable groups. All these authors were inspired by a similar notion by Brin and Katok \cite{BrK}.

We prove that, when $\mu$ is a $T$-invariant Borel probability measure on $X$,
\begin{equation}\label{B-KFormula}
\underline h_\mu^{loc}(T,\Gamma)\leq\th_\mu(T,\Gamma)
\end{equation}
 (see Theorem~\ref{Proposition6.2}) in the spirit of the Brin-Katok Formula. It remains an open problem to verify whether equality holds  (see Question~\ref{Ques99}), under suitable hypotheses, as it is known for the classical case of countable amenable group actions studied by Zheng and Chen \cite{ZC}. A positive response to this problem would ensure one of the inequalities in the Variational Principle \eqref{VPintro} in view of the following argument.

In the classical setting of countable amenable group actions a local version of the Variational Principle was established by Zheng and Chen \cite{ZC}, by using a version of topological entropy inspired by the one by Bowen \cite{Bow2}. So, we conjecture that also in our setting, possibly under suitable hypotheses, a ``local Variational Principle'' holds (see Conjecture~\ref{lVPconj}):
if $G$ is a semigroup, $\Gamma=(N_n)_{n\in\Z_+}$ a regular system in $G$, $T\colon G\times X\to X$ a continuous action on a compact metric space $X$, then 
\begin{equation}\label{lVPintro}
\th(T,\Gamma)=\sup\{\underline h_\mu^{loc}(T,\Gamma)\colon \mu\ \text{Borel probability measure on $X$}\}. 
\end{equation}
This was proved for $\Z_+$-actions by Feng and Huang  \cite[Theorem 1.2]{FH}. On the other hand, the inequality $\geq$ in \eqref{lVPintro} can be proved by applying Theorem~\ref{Theorem5.3} and Theorem~\ref{Theorem5.5}, 
as follows. In Theorem~\ref{Proposition6.1} we show the intermediate inequality $\th(T,\Gamma)\geq \underline S_{\mu}(T,\Gamma)$,  
where $\underline S_\mu(T,\Gamma)$ denotes the essential supremum of the measurable function $h_\mu^{loc}\colon X\to \R_{\geq0}$. As 
$\underline S_{\mu}(T,\Gamma) \geq \underline h_\mu^{loc}(T,\Gamma)$, we get half of the local Variational Principle, namely, 
the inequality $\geq$ in \eqref{lVPintro}. 
Nevertheless, also in this case the problem of the validity of the local Variational Principle \eqref{lVPintro} remains open. 

The conjecture on the validity of the local Variational Principle in \eqref{lVPintro} is supported also by a similar result obtained in \cite{KC} for the so-called slow entropy. Indeed, starting from the above-mentioned problem that the classical metric and topological entropy vanish very often, Katok and Thouvenot \cite{KT} proposed a notion of slow metric entropy for $\Z^d$-actions (see also Hochman \cite{Hochman}.
A slow topological entropy inspired by Bowen's topological entropy from \cite{Bow2} was introduced by Kong and Chen \cite{KC} for $\Z^d$-actions on compact metric spaces (note that the slow topological entropy does not coincide with the classical one even in the case $d=1$). In the same setting, Kong and Chen introduced also a notion of slow lower local metric entropy (which does not coincide with that in \cite{KT}) and proved that the counterpart of the local Variational Principle for the slow entropies holds.

\medskip
We warmly thank the referees for their careful reading and useful comments and suggestions.

\subsection{Notation and terminology}

 For a set $X$ we denote by $\covv(X)$ the {\em set of covers} of $X$ and by $\paa(X)$ the {\em set of partitions} of $X$. 
For $\ee_1,\ee_2, \ldots, \ee_k\in \covv(X)$ let 
$$\bigvee_{i=1}^k \ee_i=\{E_1\cap E_2\cap\ldots \cap E_k\colon E_i\in \ee_i\ \forall i\in\{1,2,\ldots,k\}\}.$$
If $\ee_1,\ee_2, \ldots,\ee_k\in \paa(X)$, then $\bigvee_{i=1}^k \ee_i\in \paa(X)$ as well. 
For a topological space $X$ we denote by $\cov(X)$ the {\em subfamily of $\covv(X)$ consisting of open covers} of $X$. If $X$ is a measure space with a probability measure we denote by $\pa(X)$ the {\em subfamily of $\paa(X)$ consisting of finite  partitions} of $X$ by means of measurable subsets.

The set $\covv(X)$ carries a natural preorder: if $\ee$ and $\aa$ are families of subsets of  a set $X$ we write $\aa  \prec \ee$ if $\ee$ {\em refines} $\aa$, i.e., every $E \in \ee$ is contained in some element of $\aa$.  If  $\ee$ and $\aa$ are partitions of $X$, then $\aa  \prec \ee$ implies that every $A \in \aa$ is a union of elements of $\ee$.

 If  $T\colon G \times X\to X,\ T(g,x) = g\, x ,$ is an action of a semigroup $G$ on a set $X$, for $g\in G$ we use the notation $g$  also for the selfmap $g\colon x\mapsto gx$ of $X$ and in such a sense we write $g^{-1}A$ to denote the inverse image of a subset $A$ of $X$ under this selfmap.  For $g\in G $ and $\aa\in \covv(X)$ we let $g^{-1} \aa = \{g^{-1} A \colon A \in \aa\}$.

\begin{remark}
In this sense, when $G$ acts on a topological space $X$ by continuous  self-maps, this action induces also an action of $G$ on the lattice of open sets of $X$ by $U\mapsto g^{-1}U$ and consequently also on $\cov(X)$. 
 
Furthermore, when  $G$ acts on a probability space $X$ by {measure-preserving self-maps}, then  $G$ acts also on the Boolean $\sigma$-algebra of measurable sets of $X$, by $B\mapsto g^{-1}B$,  and consequently also on $\pa(X)$.
\end{remark}

 We consider only semigroups $G$ with identity $e$ and actions $T\colon G\times X\to X$ such that $T(e,-)=id_X$.

\subsection{Regular systems}

\begin{definition}
A \emph{regular system} in a semigroup $G$ is by definition a sequence of subsets of $G$ of the form
$\Gamma = (N_0, N_{1}, \ldots ,N_{n}, \ldots )$
such that 
\begin{equation}\label{EqRedSys}
e\in N_0 \quad \text{and} \quad N_i N_j \subseteq N_{i+j}\ \text{for all}\ i,j\in\Z_+.
\end{equation}
Throughout this paper we assume that \emph{all sets $N_n$ are finite} in a regular system $\Gamma=(N_n)_{n\in\Z_+}$.
\end{definition}


\begin{example} 
A prominent example of regular system is obtained by taking in \eqref{EqRedSys} a finite set $N_1$ of generators of a finitely generated semigroup $G$ with $N_0 \subseteq N_1$ (and so $e\in N_1$) and $N_m=N_1^m$ for all $m\geq 1$. 
Following \cite{HS}, we call such regular systems {\em standard}.  
Clearly, {only finitely generated semigroups may have standard regular systems}, in case $G = \bigcup_{i\in\Z_+} N_i$ is imposed. 

Our main example is  $G = {\Z}^{k}$ (or $G = \Z^k_+$) with a {standard regular system} given, for $n\in \Z_+$, by  $N_n = [-n,n]^k$ (respectively, $N_n = [0,n]^k$).
\end{example}

When $T\colon G\times X\to X$ is an action of a semigroup $G$ on a set $X$, given a regular system $\Gamma=(N_n)_{n\in\Z_+}$ in $G$ and  a cover $\aa$ of $X$, define for every $n\geq1$ a  cover of $X$ by 
$$\aa^n_\Gamma = \bigvee_{g \in N_n} g^{-1}\aa = \left\{ \bigcap_{g\in N_n} g^{-1} A_{g} \colon  A_{g} \in \aa \mbox{ for  all } g\in N_n
\right\}.$$
Clearly, $\aa^n_\Gamma$ is a partition whenever $\aa$ is a partition. {If the regular system is fixed we briefly write $\aa^n$ instead of $\aa^n_\Gamma.$}

\section{Metric entropy}\label{recmetric}

\subsection{Definition of the receptive metric entropy}

Assume that $(X, \mm, \mu)$ is a probability space with a probability measure $\mu$ defined on the $\sigma$-algebra $\mm$ of subsets of $X$, and let 
\begin{equation}\label{2.3}
T\colon G \times X\to X, \quad T(g,x) = g\, x ,
\end{equation}
be a \emph{measure-preserving action} of $G$ on $X$, i.e., the map $g\colon X\to X$, $x\mapsto gx$, is measure-preserving for all $g \in G$. 
For a fixed regular system $\Gamma=(N_n)_{n\in\Z_+}$ in $G$, we define the \emph{metric entropy} $\th_\mu(T) = \th_\mu(T,\Gamma)$ of $T$
with respect to $\Gamma$ essentially using the approach in \cite{HS}. Such definition in the case $G = \Z^2$ was given in the honours thesis  (diploma work) \cite{T} of Leon Todorovich, however without much follow up.

Recall (see e.g. \cite[Chapter 4]{Wa}) that if $\aa = \{A_1, A_2, \ldots, A_k\}\in \pa(X)$, 
then the \emph{entropy of $\aa$} is defined by
$$H_\mu(\aa) = - \sum_{i=1}^k \mu(A_i) \, \log \mu(A_i),$$
with the usual agreement that $0\, \log 0 = 0$. Since the measure $\mu$ is $T$-invariant, it follows that $H_\mu (g^{-1} \aa) = H_\mu(\aa)$ for all $g \in G$.

\begin{definition}
Let $G$ be a semigroup, $\Gamma=(N_n)_{n\in\Z_+}$ a regular system in $G$, and $(X,\mathcal M,\mu)$ a probability space.
The \emph{receptive metric entropy} of the measure-preserving action $T\colon G\times X\to X$ with respect to $\aa\in\pa(X)$ and $\Gamma$ is
\begin{equation}\label{2.4}
\th_{\mu}(T, \aa, \Gamma) = \limsup_{n\to\infty} \frac{1}{n} H_\mu \left(\aa^n_\Gamma\right).
\end{equation}
The \emph{receptive metric entropy of $T$} with respect to $\Gamma$ is $\th_\mu(T,\Gamma) = \sup \{\th_{\mu} (T,\aa,\Gamma)\colon \aa\in\pa(X) \}.$
\end{definition}

When the system $\Gamma$ is clear from the context and no confusion is possible, we shall omit $\Gamma$.

\begin{remark}\label{R2.4}
In the classical case $G = \Z_+$, consider the measure-preserving action $T\colon G\times X\to X$ on the probability space $(X,\mathcal M,\mu)$ as a single measure-preserving transformation $T \colon X\to X$ (i.e., $T(1,-)=T$). With respect to the standard regular system $\Gamma=([0,n])_{n\in\Z_+}$, 
$$\th_{\mu}(T, \aa) = \lim_{n\to\infty} \frac{1}{n} H_\mu \left(\bigvee_{i =0}^{n} T^{-i}\aa \right) = \inf_{n\geq 1} \frac{1}{n} H_\mu \left(\bigvee_{i =0}^{n} T^{-i}\aa \right) = h_\mu(T,\aa),$$
so $\th_\mu(T)$ coincides with the Kolmogorov-Sinai metric entropy $ h_\mu(T)$ of $T$.
\end{remark}

\begin{example}\label{rec0}
The trivial action $T\colon G\times X\to X$ of a semigroup $G$ on a probability space $(X,\mathcal M,\mu)$ has always  $\th_\mu(T) = 0$ with respect to any regular system $\Gamma=(N_n)_{n\in\Z_+}$ in $G$. Indeed, for every $\aa\in\pa(X)$ one has $\aa_\Gamma^n = \aa$, so the limit in \eqref{2.4} is $0$, hence $\th_\mu(T,\aa)=0$ regardless  of $\Gamma$. This gives $\th_\mu(T) = 0$. Similarly, $\th_\mu(T) = 0$ when $G$ is finite. 
\end{example}

\begin{remark}
The computation of the receptive measure entropy can be reduced to the case of faithful actions. Indeed, the measure-preserving action $T\colon G\times X \to X$ of a semigroup $G$ on a measure probability space $(X,\mathcal M,\mu)$ can be factorized in an obvious way through an action $T'\colon G' \times X \to X$ of a quotient $G'$ of $G$ that faithfully acts on $X$ (i.e., if $g x = x$ for all $x\in X$, then $g = e_{G'}$). Since, given a regular system $\Gamma=(N_n)_{n\in\Z_+}$ in $G$, for every $n\geq1$, $\aa^n_\Gamma$ coincides for both actions and for every $\aa\in\pa(X)$,  we deduce that $\th_\mu(T,\aa)= \th_\mu(T',\aa)$. 
\end{remark}

It follows immediately from the definition that the receptive metric entropy is a \emph{conjugacy invariant}.  This comes also from the following more general property of monotonicity (obtained with $H=G$ and $\eta = id_G$): 

\begin{proposition}\label{bis}
Let $G$ be a semigroup, $\Gamma=(N_n)_{n\in\Z_+}$ a regular system in $G$ and $(X,\mathcal M,\mu)$ a probability space. Let  $(Y,\nn, \nu)$ be another probability space, let $S\colon H \times Y\to Y$ be a measure-preserving action and let there exist a measure-preserving surjective map $\varphi\colon X\to Y$ and a homomorphism $\eta\colon G\to H$ with $S(\eta(g), \varphi(x)) = \varphi (T(g,x))$  for all $g \in G$ and $x \in X$ (briefly, $\eta(g)\circ\varphi=\varphi\circ g$ for every $g\in G$). 
$$\xymatrix{ X\ar[d]_{\varphi} \ar[r]^{g}& X\ar[d]^{\varphi}\\
Y \ar[r]_{\eta(g)} & Y}$$
Then $\th_\nu(S,\eta(\Gamma)) \leq \th_\mu(T,\Gamma)$ with respect to a fixed regular system $\Gamma=(N_n)_{n\in\Z_+}$ in $G$. If $\eta$ is an isomorphism and $\varphi$ is an isomorphism (i.e., $\varphi$ is a bijection),  then $\th_\nu(S,\eta(\Gamma)) = \th_\mu(T,\Gamma)$.
\end{proposition}
\begin{proof} 
Given $\bb = \{ B_1,B_2, \ldots, B_k\}\in\pa(Y)$, $\varphi^{-1}(\bb) = \{\varphi^{-1}(B_1),\varphi^{-1}(B_2), \ldots,\varphi^{-1}(B_k)\}\in\pa(X)$ with $\mu(\varphi^{-1}(B_i)) = \nu(B_i)$ for all $i\in\{1,2,\ldots,k\}$, as $\varphi$ is measure-preserving. Fixed $n \geq 1$, we verify that
\begin{equation}\label{==}
(\varphi^{-1}( \bb))_\Gamma^n= \varphi^{-1}\left(\bb^n_{\eta(\Gamma)}\right).
\end{equation}
To this end, take $B\in\varphi^{-1}\left(\bb^n_{\eta(\Gamma)}\right)$, i.e., $B:=\varphi^{-1}\left(\bigcap_{h\in\eta(N_n)}h^{-1}B_h\right)$
with $B_h\in\bb$ for every $h\in \eta(N_n)$. 
Then 
$$B=\bigcap_{h\in\eta(N_n)}\varphi^{-1}(h^{-1}B_h)=\bigcap_{g\in N_n}\varphi^{-1}(\eta(g)^{-1}B_{\eta(g)})=\bigcap_{g\in N_n}g^{-1}\varphi^{-1}(B_{\eta(g)})\in (\varphi^{-1}( \bb))_\Gamma^n,$$
where for each $h\in \eta(N_n)$ an element $g\in N_n$ is chosen with $\eta(g) = h$. This proves the inclusion $\varphi^{-1}\left(\bb^n_{\eta(\Gamma)}\right)\subseteq (\varphi^{-1}( \bb))_\Gamma^n$. Since these are partitions, we conclude that \eqref{==} holds.
Since $\varphi$ is measure-preserving, this implies $$H_\nu \left(\bb^n_{\eta(\Gamma)}\right) = H_\mu \left(\varphi^{-1}(\bb^n_{\eta(\Gamma)})\right)\leq H_\mu \left((\varphi^{-1}( \bb))_\Gamma^n\right),$$ so $\th_\nu (S, \bb,\eta(\Gamma)) \leq \th_\mu(T, \varphi^{-1}(\bb),\Gamma)\leq \th_\mu(T,\Gamma)$. Therefore $\th_\nu(S, \eta(\Gamma)) \leq \th_\mu(T,\Gamma)$.

\smallskip
If $\eta$ is an isomorphism and $\varphi$ is an isomorphism too, taking $\varphi^{-1}$ and $\eta^{-1}$ in place of $\varphi$ and $\eta$, and $\eta(\Gamma)$ in place of $\Gamma$ (now $\eta^{-1}(\eta(\Gamma))=\Gamma$), the above argument gives $\th_\mu(T,\Gamma) \leq \th_\nu(S,\eta(\Gamma))$. So $\th_\mu(T,\Gamma) = \th_\nu(S,\eta(\Gamma))$.
\end{proof}

\subsection{First properties of the receptive metric entropy}

\begin{proposition}\label{Proposition2.6}
Let $G$ be a semigroup, $\Gamma=(N_n)_{n\in\Z_+}$ a regular system in $G$, $(X,\mathcal M,\mu)$ a probability space and $T\colon G\times X\to X$ a measure-preserving action.
Let $p\geq 1$ be an integer and $\Gamma' =(N_{pn})_{n\in\Z_+}$. 
Then $\th_\mu(T,\aa,\Gamma') =  p \cdot \th_\mu(T,\aa,\Gamma)$ for every $\aa\in\pa(X)$, and thus $\th_\mu(T,\Gamma') = p\cdot \th_\mu(T,\Gamma).$
\end{proposition}
\begin{proof}
Given $\aa\in\pa(X)$,  it is enough to prove that
\begin{equation}\label{2.7}
\th_\mu(T, \aa, \Gamma):=\limsup_{n\to\infty} \frac{1}{n} H_\mu \left(\bigvee_{g \in N_n} g^{-1}\aa \right) 
= \limsup_{n\to\infty} \frac{1}{n p } H_\mu \left(\bigvee_{g \in N_{pn}} g^{-1}\aa \right)=:\frac{1}{p}\th_\mu(T,\aa,\Gamma').
\end{equation}
Since $\frac{1}{np} H_\mu(\bigvee_{g\in N_{np}}g^{-1}A)$ is a subsequence of $\frac{1}{n} H_\mu(\bigvee_{g\in N_n}g^{-1}A)$, we immediately get the inequality 
$$\limsup_{n\to\infty} \frac{1}{n} H_\mu \left(\bigvee_{g \in N_n} g^{-1}\aa \right) \geq \limsup_{n\to\infty} \frac{1}{np} H_\mu \left(\bigvee_{g \in N_{np}} g^{-1}\aa \right).$$
To prove the converse inequality, note that there exists an increasing sequence of integers $(n_r)_{r\in\Z_+}$ with $n_r\to \infty$ and 
such that 
\begin{equation}\label{submission}
\limsup_{n\to\infty} \frac{1}{n} H_\mu \left(\bigvee_{g \in N_n} g^{-1}\aa \right) = \lim_{r\to\infty} \frac{1}{n_r} H_\mu \left(\bigvee_{g \in N_{n_r}} g^{-1}\aa \right).
\end{equation}
For every  $r \geq p$ put $m_r := \lfloor \frac{n_r}{p}\rfloor \geq 1$, so that
$$m_r p \leq n_r < (m_r +1) p .$$
Then 
$$\frac{1}{n_r} H_\mu \left(\bigvee_{g \in N_{n_r}} g^{-1}\aa \right) \leq  \frac{1}{n_r} H_\mu \left(\bigvee_{g \in N_{(m_r+1)p}} g^{-1}\aa \right) 
= \frac{(m_r+1)p}{n_r} \cdot \frac{1}{(m_r+1)p} H_\mu \left(\bigvee_{g \in N_{(m_r+1)p}} g^{-1}\aa \right) , $$
which, together with \eqref{submission}, implies
$$\limsup_{n\to\infty} \frac{1}{n} H_\mu \left(\bigvee_{g \in N_n} g^{-1}\aa \right)=\lim_{r\to\infty} \frac{1}{n_r} H_\mu \left(\bigvee_{g \in N_{n_r}} g^{-1}\aa \right)  \leq  \limsup_{n\to\infty} \frac{1}{np} H_\mu \left(\bigvee_{g \in N_{np}} g^{-1}\aa \right).$$
Hence \eqref{2.7} holds.
\end{proof}

For a slight generalization of this property see \cite[Remark 2.3(b)]{B-S1}. 

For $G =\Z_+$ the equality $\th_\mu(T,\Gamma') = p\cdot \th_\mu(T,\Gamma)$ is known as ``logarithmic law'', since it becomes $h_\mu(T^p)=p\cdot h_\mu(T)$, where $T$ is the measure-preserving transformation $T(1,-)$ and $\Gamma=([0,n])_{n\in\Z_+}$.

\medskip
 It is natural to compare the metric entropy of each individual map $g\in G$ with $\th_\mu(T)$. It turns out that for elements $g \in N_1$ the metric entropy of $g$ is bounded above by $\th_\mu(T)$.

\begin{proposition}\label{Proposition2.3} 
Let $G$ be a semigroup, $\Gamma=(N_n)_{n\in\Z_+}$ a regular system in $G$, $(X,\mathcal M,\mu)$ a probability space and $T\colon G\times X\to X$ be a measure-preserving action. If $g \in N_m$ for some integer $m\geq 1$, then $h_\mu(g) \leq m\cdot \th_\mu(T)$. Thus,
\begin{equation}\label{2.6}
\sup_{g\in N_1} h_\mu(g) \leq \th_\mu(T) .
\end{equation}
\end{proposition}
\begin{proof}
Let $g \in N_m$.  Given $\aa\in\pa(X)$, Remark~\ref{R2.4} yields 
$$h_\mu(g,\aa) =  \lim_{n\to \infty} \frac{1}{n} H_\mu \left(\bigvee_{i =0}^{n} g^{-i}\aa \right).$$
Since $g^i \in N_{n m}$ for all $i\in\{0,1, \ldots, n\}$, it follows that 
$$ \bigvee_{i =0}^{n} g^{-i}\aa\prec\bigvee_{h\in N_{mn}} h^{-1}\aa.$$
 By well-known properties of partitions (see e.g. \cite[Theorem 4.3]{Wa}), $\bb \prec \cc$ in $\pa(X)$ entails $H_\mu (\bb) \leq H_\mu (\cc)$. 
Thus, letting $\Gamma'=(N_{mn})_{n\in\Z_+}$ and applying Proposition~\ref{Proposition2.6} in the last equality, we deduce that 
$$h_\mu(g,\aa) 
\leq   \limsup_{n\to \infty} \frac{1}{n} H_\mu \left(\bigvee_{h\in N_{mn}} h^{-1}\aa \right)=\th_\mu(T,\aa,\Gamma') =m\cdot\th_\mu(T,\aa, \Gamma).$$
Therefore $h_\mu(g) \leq m\cdot \th_\mu(T)$.
\end{proof}

\begin{remark}
The receptive metric entropy of a measure-preserving action $T\colon G\times X\to X$ as above depends on the choice of the regular system $\Gamma$. One can be tempted to get rid of this dependence by taking 
\begin{equation}\label{AQ}
\sup \{\th_\mu(T,\Gamma)\colon\Gamma \ \mbox{ is a regular system for} \ G\}. 
\end{equation}
Proposition~\ref{Proposition2.6} shows that \eqref{AQ}  cannot be  a good remedy, as it can take only the values $0$ and $\infty$.  
\end{remark}

\subsection{Classical metric entropy}\label{classicalhm}

Recall that a discrete countably infinite cancellative semigroup $G$ is called \emph{left amenable}  if it admits a \emph{F\o lner sequence}, i.e., a sequence $(F_n)_{n\in\Z_+}$ of finite subsets of $G$ such that 
$$\lim_{n\to\infty} \frac{|F_n \; \Delta \; gF_n|}{|F_n|} = 0.$$
For countably infinite groups left and right amenability coincide. Finite semigroups are obviously amenable. 

The classical definition of \emph{metric entropy of a measure-preserving amenable semigroup action} $T\colon G \times X\to X$ on a probability space $(X,\mathcal M,\mu)$, such that $(N_n)_{n\in\Z_+}$ is a F\o lner sequence, is 
$$h_\mu(T) = \sup\{h_{\mu} (T,\aa)\colon \aa\in\pa(X)\},$$
where 
$$h_\mu(T, \aa) = \lim_{n\to\infty} \frac{1}{|N_n|} H_\mu \left(\bigvee_{g \in N_n} g^{-1}\aa \right).$$
See \cite{Con,Ru1,Ru2,E}. 
In this case the limit is known to exist (see e.g. \cite[Theorem 2.1]{Con}).

\begin{example}
When $G$ is a finite semigroup, for a measure-preserving action $T\colon G\times X\to X$  of $G$ on the probability space $(X,\mathcal M,\mu)$, $h_\mu(T,\aa)>0$ for every $\aa\in\pa(X)$ and one may even get $h_\mu(T)= \infty$ if $X$ has partitions $\aa$ with arbitrarily large $H_\mu(\aa)$.
This should be compared with the value $0$ of $\th_\mu(T)$ when $G$ is  finite (see Example~\ref{rec0}).
\end{example}

For $G =\Z_+^k$ every standard regular system gives rise to a F\o lner sequence (see \cite{B-S1}), while in $G = \Z_+$, $(N_n)_{n\in\Z_+}$ with $N_n =\{0,2,4,\ldots,2n\}$ for every $n\in\Z_+$, is a regular system which is not a F\o lner sequence.

\begin{remark}
Let $G$ be a semigroup and suppose that the regular system $\Gamma=(N_n)_{n\in\Z_+}$ in $G$ is a {F\o lner} sequence and satisfies  
\begin{equation}\label{growth}
|N_n| \geq c n^2\ \mbox{ for some constant }\  c > 0\ \mbox{ and all }\ n \geq 1.
\end{equation}
Let $T\colon G\times X\to X$ be a measure-preserving action on a probability space $(X,\mathcal M,\mu)$.
It then follows from the definitions that if for some $\aa\in\pa(X)$ we have $\th_\mu(T,\aa) < \infty$, then $h_\mu(T,\aa) = 0$. Thus, $\th_\mu(T) < \infty $ implies $h_\mu(T) = 0$. Equivalently, if $h_\mu(T) > 0$, then $\th_\mu(T) = \infty$.

The condition \eqref{growth} is available whenever $G$ is a finitely generated group that is not commensurable with a cyclic group (e.g., $G= \Z^k$ 
for some $k \geq 2$, see \cite{B-S1} for more detail).
Recall that two subgroups $H_1$ and $H_2$ of a group $G$ are said to be commensurable if the subgroup $H_1\cap H_2$ has finite index both in $H_1$ and $H_2$.
\end{remark}

In contrast with Proposition~\ref{Proposition2.3}, in the classical case, as mentioned  for example in \cite{Con}, $h_\mu(T) \leq h_\mu(g)$ for any $g \in G$. More importantly, we have the following fact proved by Conze \cite[Theorem 2.3]{Con} (where the case $k =2$ is dealt with, the general case is proved in a  similar way).

\begin{proposition}\label{Proposition2.4} 
Let $G = \Z^k$, $k >1$, let $g_1,g_2, \ldots,g_k$ be generators of $G$, and let $T\colon G\times X\to X$ be a measure-preserving action of $G$ on a probability space $(X,\mathcal M,\mu)$. If $h_\mu(g_i) < \infty$ for some $i\in\{1,2, \ldots, k\}$, then $h_\mu(T) = 0$.
\end{proposition}


The next example shows that unlike the classical  metric entropy the receptive one does not vanish in this situation.

\begin{example}\label{Example2.5} 
Let $f\colon X\to X$ be a measure-preserving transformation of the probability space $(X,\mm, \mu)$, and let $G = \Z^k_+$ with $k > 1$ and the standard regular system $\Gamma=(N_n)_{n\in\Z_+}$ with $N_n = [0,n]^k$.
Define the action $T\colon G \times X\to X$ by 
$$T((m_1,m_2, \ldots, m_k), x) = f^{m_1+m_2+ \ldots + m_k}(x).$$
Assume $h_\mu(f) < \infty$. It then follows from Proposition~\ref{Proposition2.4} that $h_\mu(T) = 0$. 
On the other hand, we compute that $\th_\mu(T) = k\cdot h_\mu(f)$. Indeed, given $\aa\in\pa(X)$, for every $n\geq 1$,
$$\bigvee_{(m_1,m_2,\ldots,m_k) \in N_n} f^{-m_1- m_2-\ldots - m_k} \aa = \bigvee_{0 \leq m \leq kn} f^{-m}\aa,$$
so
$$\frac{1}{n} H_\mu \left(\bigvee_{(m_1,m_2,\ldots,m_k) \in N_n} f^{-m_1- m_2-\ldots - m_k} \aa \right)=\frac{1}{n} H_\mu \left(\bigvee_{0 \leq m \leq k n} f^{-m}\aa\right) =  k\cdot \frac{1}{kn} H_\mu \left(\bigvee_{0 \leq m \leq k n} f^{-m}\aa\right),$$
and taking limits as $n \to \infty$ (the limit exists in the right-hand side by Remark~\ref{R2.4}, so it exists in the left-hand side as well),  we get $\th_\mu(T, \aa) = k\cdot h_\mu(f, \aa)$. Hence $\th_\mu(T) = k\cdot h_\mu(f) < \infty$.

\smallskip
To get a particular example, let $X = \prod_{-\infty}^\infty  \{ 1,2,\ldots,r\}$, let $f=\sigma\colon X\to  X$ be the Bernoulli shift on $r$ symbols with probabilities $p_1,p_2, \ldots, p_r > 0$ and $\sum_{i=1}^r p_i = 1$, and let $\mu$ be the corresponding invariant (product) measure on $X$. Then $h_\mu(f) = - \sum_{i=1}^r p_i\, \log p_i < \infty$ (see e.g. \cite[Theorem 4.26]{Wa}). Then for the action $T$ defined above it follows that $\th_\mu(T) =  - k\, \sum_{i=1}^r p_i\, \log p_i < \infty$.
\end{example}

\section{Further properties of the receptive metric entropy}\label{further}

We begin with some facts about subactions. For $T\colon G\times X\to X$ a measure-preserving action of the semigroup $G$ on the probability space $(X,\mm, \mu)$, and $H$ a subsemigroup of $G$,  let $T\restriction_H\colon H \times X\to X$ be the action of $H$ on $X$ induced by $T$.

\begin{proposition}\label{Proposition3.1} 
Let $G$ be a semigroup, let  $\Gamma=(N_n)_{n\in\Z_+}$ be a regular system in $G$, and let $T\colon G\times X\to X$ be a measure-preserving action of $G$ on the probability space $(X,\mm, \mu)$. Let $H$ be a subsemigroup of $G$, and let $S=T\restriction_H$. Define the regular system $\Gamma _H= (M_n)_{n\in\Z_+}$ in $H$, by $M_n = N_n \cap H$ for all $n \geq 0$.  Then:
\begin{itemize}
\item[(a)] $\th_\mu(S,\Gamma_H) \leq \th_\mu(T,\Gamma )$;
\item[(b)] for $G = \Z^k_+$ with the standard regular system given by $N_n = [0,n]^k$ and its subsemigroup 
$H = p_1\Z_+ \times p_2\Z_+ \times \ldots \times p_k \Z_+$, where $p_1,p_2, \ldots, p_k \geq 1$ are given integers, $\th_\mu(S,\Gamma_H) \leq\th_\mu(T,\Gamma) \leq (p_1 p_2 \ldots p_k)\cdot \th_\mu (S,\Gamma_H).$
\end{itemize}
\end{proposition}
\begin{proof} 
(a) Clearly, for $\aa\in\pa(X)$ and for all $n \geq 1$, $
\frac{1}{n} H_\mu \left( \bigvee_{u \in M_n} u^{-1}\aa\right) \leq \frac{1}{n} H_\mu \left( \bigvee_{g \in N_{n}} g^{-1}\aa \right)$,
and so $\th_\mu(S, \aa) \leq \th_\mu(T, \aa)$.

(b) The first inequality is item (a).  Without loss of generality we may assume that $1\le p_1 \leq p_2 \leq \ldots \leq p_k$. Let  $F = \prod_{i=1}^k \{0,1, \ldots, p_i-1 \},$ with $|F| = p_1 p_2 \ldots p_k$. Let $n \geq p_k$. Then $N_n = M_n + F$. Indeed, for any $g = (m_1,m_2, \ldots, m_k) \in N_n$ and for every $i\in\{1, 2,\ldots, k\}$ we can write $m_i = t_i p_i + s_i$ for some integers $t_i \geq 1$ and $s_i\in\{0,1,\ldots, p_i -1\}$.
Thus, $g = u + s$, where $u = (t_1p_1,t_2p_2, \ldots, t_k p_k) \in M_n$ and $s = (s_1,s_2, \ldots, s_k) \in F$. 
The above argument implies that, for $\aa\in\pa(X)$ and an integer $n\geq 1$,
$$\bigvee_{g \in N_n}-g + \aa= \bigvee_{s\in F, u \in M_n} -(u+s) + \aa =\bigvee_{s\in F} -s + \left(\bigvee_{u\in M_n}-u + \aa\right),$$
and therefore
$$H_\mu\left(\bigvee_{g\in N_n}-g+\aa\right)\leq\sum_{s\in F}H_\mu\left(-s+\bigvee_{u\in M_n}-u+\aa \right)=\sum_{s\in F}H_\mu\left(\bigvee_{u\in M_n}-u+\aa\right),$$
the last equality due to the fact that each $s\in F$ is measure-preserving.  Thus, for all $n \geq 1$,
$$\frac{1}{n} H_\mu\left( \bigvee_{g \in N_n}-g + \aa\right)  \leq |F|\; \frac{1}{n}  H_\mu\left(\bigvee_{u\in M_n}- u + \aa\right),$$
which gives $\th_\mu(T, \aa,\Gamma) \leq |F|\cdot \th_\mu(S, \aa,\Gamma_H)$. Hence $\th_\mu(T,\Gamma)\leq |F|\cdot \th_\mu(S,\Gamma_H)$.
\end{proof}

Now we make a different choice of the regular system in a specific case. 

\begin{lemma}\label{(b)}
Let $G$ be a commutative semigroup, $\Gamma=(N_n)_{n\in\Z_+}$ a regular system in $G$, and let $T\colon G\times X\to X$ be a measure-preserving action of $G$ on the probability measure space $(X,\mm, \mu)$. Let $p \geq 1$, $H=pG$, and let $S=T\restriction_H$. Consider $G$ with the regular system 
$p\Gamma=(pN_n)_{n\in\Z_+}$ and $H$ with the regular system $\Gamma_H=(N_{n} \cap H)_{n\in\Z_+}$.
Then $\th_\mu(T, p\Gamma)=\th_\mu(S, p\Gamma)\leq p\, \th_\mu(S, \Gamma_H)$.
\end{lemma}
\begin{proof} 
Let $\aa\in\pa(X)$. For every $n \geq 1$ and every $g \in N_n$ we have $pg \in p N_n \subseteq N_{pn} \cap H$, so 
$$\bigvee_{u \in pN_n}u^{-1}\aa\prec \bigvee_{g \in N_{pn} \cap H}g^{-1}\aa .$$
Thus, for all $n \geq 1$,
$$\frac{1}{n} H_\mu \left( \bigvee_{u \in pN_n} u^{-1}\aa \right) \leq \frac{1}{n} H_\mu \left(\bigvee_{g\in N_{pn} \cap H}g^{-1}\aa \right) \leq
p\cdot \frac{1}{pn} H_\mu \left(\bigvee_{g\in N_{pn} \cap H}g^{-1}\aa \right).$$
Obviously, 
$$\limsup_{n\to \infty}  \frac{1}{pn} H_\mu \left(\bigvee_{g\in N_{pn} \cap H}g^{-1}\aa \right) \leq  \limsup_{n\to \infty}  \frac{1}{n} H_\mu \left(\bigvee_{g\in N_{n} \cap H}g^{-1}\aa \right) .$$
Combining this with the above gives
\begin{align*}
\th_\mu(T,\aa,p\Gamma) 
& = \limsup_{n\to \infty}  \frac{1}{n} H_\mu \left( \bigvee_{u \in pN_n} u^{-1}\aa \right) 
\leq p \; \limsup_{n\to \infty}  \frac{1}{pn} H_\mu \left(\bigvee_{g\in N_{pn} \cap H}g^{-1}\aa \right) \\
& = p\; \limsup_{n\to \infty}  \frac{1}{n} H_\mu \left(\bigvee_{g\in N_{n} \cap H}g^{-1}\aa \right)  = p \, \th_\mu(S,\aa,\Gamma_H) \qedhere.
\end{align*}
\end{proof}

\begin{example}\label{pH}
Let $G$, $p\geq1$, $H$, $X$, $T,S$ and $\Gamma=(N_n)_{n\in\Z_+}$ be as in Lemma~\ref{(b)}. Then, $\th_\mu(T,p\Gamma)=\th_\mu(S,p\Gamma)\leq \th_\mu(T,\Gamma)$ according to 
Proposition~\ref{bis} with $Y=X$, $\varphi = id_X$ and with $\eta$ the homomorphism $\eta:G \to H$ defined by 
$\eta(g)=pg$ for $g\in G$. By the final assertion of that proposition, $\th_\mu(S, p\Gamma)= \th_\mu(T,\Gamma)$, provided 
$\eta$ is an isomorphism (equivalently, injective).  

Let us apply this observation to the monoid $G=\Z_+^k$ with $k\geq 1$ and standard regular system $\Gamma=(N_n)_{n\in\Z_+}$, with 
$N_n=[0,n]^k$ for every $n\in\Z_+$. Now for every $p \geq 1$  the homomorphism $\eta:G \to H$, defined above,
is injective. So for the subsemigroup $H = pG$ of $G$ and  for every measure-preserving action $T\colon G\times X\to X$ on a probability measure space $(X,\mm, \mu)$, 
the above argument gives $\th_\mu(S, p\Gamma)= \th_\mu(T,\Gamma)$  for the restriction $S$ of $T$ on $H$. Since, obviously 
$p\Gamma=\Gamma_H=(pN_n)_{n\in\Z_+}$ in this case, we can add to the above equality also $\th_\mu(T, p\Gamma)=\th_\mu(S, p\Gamma)=  \th_\mu (S,\Gamma_H) =\th_\mu(T,\Gamma)$.
This should be compared with the much less sharp inequalities from Proposition \ref{Proposition3.1}(b) and from Lemma~\ref{(b)}.
\end{example}

Next we deal with products of measure-preserving actions. The proof of the following proposition is similar to that of the corresponding result about  the case $G = \Z_+$ (see e.g. the proof of \cite[Theorem 4.23]{Wa}). {To prove Proposition~\ref{Proposition3.2} we need the following lemma.}

\begin{lemma}\label{ablogab} 
For finite partitions $\mathcal{A}=\{A_1,A_2,\ldots,A_k\}$ of a probability space $(X_1, \mathcal M_1,\mu_1)$ and $\mathcal{B}=\{B_1,B_2,\ldots, B_l\}$ of a probability space $(X_2, \mathcal M_2,\mu_2)$ let $a_i=\mu_1(A_i)$ and $b_j=\mu_2(B_j),$
for any $i\in\{1,2,\ldots,k\}$ and $j\in\{1,2,\ldots,l\}$. Then $$\sum_{i,j} a_i b_j(\log a_i+\log b_j)=  \sum_{i=1}^k a_i \log a_i +\sum_{j=1}^l b_j \log b_j.$$
\end{lemma}

\begin{proposition}\label{Proposition3.2} 
For $i=1,2$, let $G_i$ be a commutative finitely generated semigroup with a regular system $\Gamma^{(i)}= (N_{n}^{(i)})_{n\in\Z_+}$, let $(X_i, \mm_i, \mu_i)$ be a probability space and let $T_{i}\colon G_{i} \times X_{i}\to X_{i}$ be a measure-preserving action.
Consider $X = X_1 \times X_2$ with the product measure $\mu = \mu_1 \times \mu_2$ defined on the
$\sigma$-algebra $\mm$ in $X$ generated by $\mm_1 \times \mm_2$, and the semigroup $G = G_{1}\times G_{2}$ with the regular system 
$\Gamma=(N_n)_{n\in\Z_+}$ given by $N_{n} = N_{n}^{(1)}\times N_{n}^{(2)}$ for all $n\in\Z_+$. Let $T\colon G \times X \to X$ be the action defined by
$$T\left( (g_{1},g_{2}),(x_{1},x_{2})\right) = \left( T_{1}(g_{1},x_{1}),T_{2}(g_{2},x_{2})\right).$$
Then
\begin{equation}\label{3.1}
\max\{\th_{\mu_{1}}(T_{1}),\th_{\mu_{2}} (T_{2}) \} \leq  \th_{\mu}(T) \leq \th_{\mu_{1}}(T_{1}) + \th_{\mu_{2}} (T_{2}).
\end{equation}
\end{proposition}
\begin{proof} 
It is enough to consider partitions of $X$ of the form $\cc = \aa \times \bb = \{A\times B\colon A\in\aa, B\in\bb\},$
where $\aa\in\pa(X_1)$ and $\bb\in\pa(X_2)$. 
For every $n \geq 1$, 
$$\cc^n_\Gamma=  \bigvee_{g \in N_n} g^{-1} \cc = \bigvee_{(g_1, g_2) \in N_{n}^{(1)}\times N_{n}^{(2)}} (g_1,g_2)^{-1} (\aa\times \bb)
= \left(\bigvee_{g_1 \in N_{n}^{(1)}} g_1^{-1} \aa \right) \times  \left(\bigvee_{g_2 \in N_{n}^{(2)}} g_2^{-1} \bb\right) .$$
By Lemma~\ref{ablogab}, 
\begin{equation}\label{3.2}
\begin{split}
H_\mu (\cc^n_\Gamma)& = - \sum_{C\times D\in\cc^n_\Gamma}\mu(C\times D) \; \log \mu(C\times D) = - \sum_{(C,D) \in \cc^n_\Gamma} \mu_1(C) \; \mu_2(D) \; (\log \mu_1(C)+\log\mu_2(D))=\\
&= - \sum_{C \in \aa^n_{\Gamma^{(1)}}} \mu_1(C) \; \log \mu_1(C) - \sum_{D \in\bb^n_{\Gamma^{(2)}}} \mu_2(D) \;  \log\mu_2(D)=H_{\mu_1} (\aa^n_{\Gamma^{(1)}}) + H_{\mu_2} (\bb^n_{\Gamma^{(2)}}) .
\end{split}\end{equation}
Let $\th_\mu(T,\cc) = \lim_{k\to \infty}\frac{1}{n_k}\; H_{\mu} (\cc^{n_k})$ for some increasing sequence of integers $(n_k)_{k\in\Z_+}$ with $n_k\to\infty$. Then the above equalities show that
$$\th_\mu(T, \cc) \leq \limsup_{k\to\infty}  \frac{1}{n_k}\; H_{\mu_1} (\aa^{n_k}_{\Gamma^{(1)}})  + \limsup_{k\to\infty}  \frac{1}{n_k}\; H_{\mu_2} (\bb^{n_k}_{\Gamma^{(2)}})
\leq \th_{\mu_1} (T_1, \aa) + \th_{\mu_2} (T_2, \bb) .$$
This implies the second inequality in \eqref{3.1}. The first inequality in \eqref{3.1} follows immediately from \eqref{3.2}.
\end{proof}

The above is easily generalised to arbitrary finite products, so that it can be applied in the next example. 

\begin{example}\label{Example3.3} 
For $i\in\{1,2, \ldots, k\}$, let $(X_i, \mm_i, \mu_i)$ be a probability space and $f_{i} : X_{i}\to X_{i}$ a measure-preserving transformation. Endow the product space $X = X_{1} \times X_2\times \ldots \times X_{k}$ with the product measure $\mu = \mu_1 \times \mu_2\times\ldots \times \mu_k$. Consider the standard regular system $\Gamma=([0,n]^k)_{n\in\Z_+}$ in $\Z^k_+$ and define the action
$$T\colon\Z_+^{k} \times X\to X\quad \text{by}\quad T\left((n_{1},n_2,\ldots,n_{k}),(x_{1},x_2,\ldots,x_{k})\right)=\left( f_{1}^{n_{1}}(x_{1}),f_2^{n_2}(x_2),\ldots,f_{k}^{n_{k}}(x_{k})\right).$$
Then Proposition \ref{Proposition3.2} gives
$$\max\{h_{\mu_i}(f_{i})\colon i\in\{1,2,\ldots,k\}\} \leq  \th_{\mu}(T) \leq \sum_{i=1}^{k}h_{\mu_i}(f_{i}) . $$
Hence  $\th_{\mu}(T) > 0$ if $h(f_{i}) > 0$ for at least one $i\in\{1,2,\ldots,k\}$.
\end{example}

As mentioned in \cite{HS}, if the components $f_{i}$ in Example~\ref{Example3.3} are smooth maps on smooth compact manifolds with topological entropy $h(f_{i}) > 0$, then the smooth action $T$ of ${\Z}_{+}^{k}$ has positive receptive topological entropy (to be defined in the next section). The above shows that the same kind of examples exist with respect to the receptive metric entropy $\th_\mu$.

\section{Topological entropy}\label{rectop}

\subsection{Receptive topological entropy}\label{recsec}

Let $(X,d)$ be a metric space, let $T\colon G\times X\to X$ be an action of the semigroup $G$ on $X$ such that each of the maps $g\colon x \mapsto gx$ is uniformly continuous (briefly, $T$ is a uniformly continuous action).

Let $\Gamma=(N_n)_{n\in\Z_+}$ be a regular system in $G$. The topological entropy that we will consider here was defined and studied in \cite{HS} under more general assumptions (in \cite{HS} $G$ is a locally compact group and each $N_n$ is just a compact subset of $G$), using spanning and separated sets, following R. Bowen's definition \cite{Bow1} in the classical case of actions of ${\Z}_{+}$. 

\begin{definition}
Let $G$ be a semigroup, $\Gamma=(N_n)_{n\in\Z_+}$ a regular system in $G$ and $T\colon G\times X\to X$ a uniformly continuous action on a metric space $X$.
Let $K$ be a compact subset of $X$, let $\epsilon > 0$ and let $n \geq 1$.

 A subset $F$ of $K$ is called an \emph{$(n,\epsilon)$-spanning set for} $K$ (with respect to $\Gamma$) if for every $x \in K$ there exists $y \in F$ such that $d(gx,gy) \leq \ep$ for all $g \in N_n$.

{A subset $E$ of $K$} is called \emph{$(n,\epsilon)$-separated} (with respect to $\Gamma$) if for any $x,y \in E$, $x \neq y$,  there exists $g \in N_n$ such that $d(gx,gy) > \epsilon$. 
\end{definition}

These definitions can be easily understood by using the \emph{dynamic balls}, for $x\in X$, $\epsilon>0$, $n\in\Z_+$,
$$D^\Gamma_n(x, \epsilon) = \{y\in X \colon d(gx, gy) \leq \epsilon,\ \forall  g \in N_n \}.$$
When it is clear from the context we shall omit the regular system $\Gamma$.

Namely, $F$ is an $(n,\epsilon)$-spanning set for the compact subset $K$ of $X$ if $K \subseteq \bigcup_{x\in F} D_n(x, \epsilon)$, and $E$ is $(n,\epsilon)$-separated if for any $x,y \in E$ with $x \neq y$, $D_n(x,\epsilon) \cap D_n(y, \epsilon)=\emptyset$.

By the compactness of $K$, the minimal $(n,\ep)$-spanning sets and the maximal $(n,\ep)$-separated sets are finite, so we can define 
\begin{equation}\label{rsn}
\begin{split}
&r_{n}(\epsilon,T,K) = \min \left\{\vert F\vert\colon F \subseteq K,\ F\ \text{is $(n,\epsilon)$-spanning for}\ K \right\},\\
&s_{n}(\epsilon,T,K)=\max \left\{\vert E\vert\colon E \subseteq K,\ E\ \text{is $(n,\epsilon)$-separated}\right\}.
\end{split}\end{equation}
Then set
\begin{equation*}\label{rs}
\tr(\epsilon,T,K) = \limsup_{n \rightarrow \infty} \frac{1}{n} \log r_n(\ep, T,K)\quad\text{and}\quad \ts(\epsilon,T,K) = \limsup_{n \rightarrow \infty} \frac{1}{n} \log s_n(\ep,T,K).
\end{equation*}
These functions are non-increasing in $\ep$ and moreover, {following arguments  similar to those in \cite[p. 169]{Wa}, we obtain
$$\tr(\epsilon,T,K) \leq \ts(\epsilon,T,K) \leq \tr(\epsilon/2, T,K) ,$$
so  the following monotone limits exist and coincide:
$$\th (T,K,\Gamma) = \lim_{\epsilon \searrow 0} \tr(\epsilon, T,K) = \lim_{\epsilon \searrow 0} \ts(\epsilon,T,K).$$

\begin{definition}
Let $G$ be a semigroup, $\Gamma=(N_n)_{n\in\Z_+}$ a regular system in $G$ and $T\colon G\times X\to X$ a uniformly continuous action on a metric space $X$. Let $K$ be a compact subset of $X$. We call $\th(T,K,\Gamma)$ the \emph{receptive topological entropy of $T$ with respect to $K$} and the given regular system $\Gamma$ in $G$. The \emph{receptive topological entropy} of $T$ with respect to $\Gamma$ is 
$$\th (T,\Gamma) = \sup \{\th (T,K,\Gamma)\colon K \subseteq X,\ K\ \text{compact}\} .$$
\end{definition}

When the system $\Gamma$ is clear from the context and no confusion is possible, we shall omit $\Gamma$ and write simply $\th(T,K)$ and $\th(T)$.

When $X$ is compact, it is easy to see that $\th(T) = \th(T, X)$. In this case we will also use the shorter notation $r_n(\ep,T) = r_n(\ep, T, X)$, $\tr(\ep,T) = \tr(\ep, T, X)$, etc.

\begin{remark}
When $G = {\Z}_{+}$ and $\Gamma=([0,n])_{n\in\Z_+}$, the uniformly continuous action $T\colon G\times X\to X$ on a metric space $X$ is given by a uniformly continuous map $f=T(1,-)\colon X \to X$. 
Then $\th(T) = h(f)$ is the \emph{topological entropy} of the map $f$, as defined by Bowen in \cite{Bow1}.
\end{remark}

It follows from \cite[Proposition 2.6]{HS} that for the uniformly continuous action $T\colon G\times X\to X$ of the semigroup $G$ on the metric space $X$, and the regular system $\Gamma=(N_n)_{n\in\Z_+}$,
$$\sup_{g \in N_1} h(g)\leq \th(T),$$ 
where $h(g)$ is the topological entropy of the map $g\colon x \mapsto gx$ (a proof can be given in the line of Proposition~\ref{Proposition2.3} above).

\medskip
For the sake of completeness we mention the following special case of \cite[Proposition 2.7]{HS}, which is an analog of Bowen-Kushnirenko Theorem in the classical case (e.g., see \cite[Theorem 7.15]{Wa}).

\begin{proposition}\label{Proposition4.1}
 Let $X$ be a Riemannian manifold and let the continuous action $T\colon G\times X\to X$ of the semigroup $G$ be such that each of the maps $g\colon x \mapsto gx$ is smooth. Consider the standard regular system $\Gamma=(N_n)_{n\in\Z_+}$ in $G$ where $N_1$ is a finite generating subset of $G$.  Consider $X$ with the metric generated by the Riemannian metric of $X$, and let $k=\dim X$. Then
$$\th(T) \leq \max \{0, k\log a\},$$
where $a = \sup_{x \in X} \max_{g \in N_1} \| d_{x}\: g \|$ and $d_x\: g$ is the differential of $g$ at $x\in X$.
\end{proposition}

\subsection{Classical topological entropy}\label{classicaltop}

Let $G$ be a semigroup and $T\colon G\times X\to X$ a uniformly continuous action on a metric space $X$.
Assume in addition that $G$ is amenable and let $\Gamma=(N_n)_{n\in\Z_+}$ be a F\o lner sequence in $G$. 
For a compact subset $K$ of the metric space $X$, define $r_n(\ep,T,K)$ and $s_n(\ep,T, K)$ as in \eqref{rsn} and 
set $$r(\epsilon,T,K) = \limsup_{n \rightarrow \infty} \frac{1}{|N_n|} \log r_n(\ep, T,K)\quad\text{and}\quad s(\epsilon,T,K) = \limsup_{n \rightarrow \infty} \frac{1}{|N_n|} \log s_n(\ep,T,K).$$
The \emph{topological entropy of $T$} with respect to $K$ is then 
$$h (T,K) = \lim_{\epsilon \searrow 0} r(\epsilon, T,K) = \lim_{\epsilon \searrow 0} s(\epsilon,T,K) $$
 while the \emph{topological entropy} is 
$$h (T) = \sup \{h(T,K)\colon K \subseteq X,\ K\ \text{compact}\}.$$
This is the definition of the classical topological entropy using spanning and separated sets (see e.g.  \cite{Ru1}).

An equivalent definition for actions on compact spaces using open covers was introduced in \cite{Oll,E}.
In fact, a more general concept, the so-called \emph{topological pressure}, has been defined similarly and studied extensively 
(see \cite{Ru1,Ru2,Oll,Mi1,El,OllP}).

\medskip
As in the case of metric entropy (see Proposition~\ref{Proposition2.4}), the classical topological entropy is zero for some  actions of $\Z^k_+$, with $k > 1$):

\begin{proposition}[See {\cite[Corollary 2.3]{E}}]
Let $G = \Z^k_+$, $k >1$, let $g_1,g_2, \ldots g_k$ be generators of $G$, and let $T\colon G\times X\to X$ a uniformly continuous action of $G$ on a metric space $X$. If $h(g_i) < \infty$ for some $i\in\{1,2,\ldots, k\}$, then $h(T)=0$.
\end{proposition}

\section{Equivalent definitions of the receptive topological entropy}\label{equiv}

\subsection{Receptive topological entropy via open covers} 

First we define the receptive topological entropy adapting the initial approach of Adler, Konheim and McAndrew \cite{AKMc}. 
Let $X$ be a compact topological space and for any $\aa\in\cov(X)$ we let $N(\aa)$ be the number of elements of a subcover of $\aa$ of the smallest possible cardinality. 

Let $T\colon G\times X\to X$ be an action of the semigroup $G$ on $X$ such that each of the maps $g\colon x \mapsto gx$ is continuous (briefly, $T$ is a continuous action), and let $\Gamma=(N_n)_{n\in\Z_+}$ be a regular system in $G$.

Given an integer $n \geq 1$ and $\aa\in\cov(X)$, set
$$\hh(T, \aa,\Gamma) = \limsup_{n \to \infty} \frac{1}{n} \log N(\aa^n_\Gamma).$$ 
Then define 
$$\hh(T,\Gamma)  = \sup\{\hh(T, \aa,\Gamma)\colon \aa\in\cov(X)\}.$$
When the system $\Gamma$ is clear from the context and no confusion is possible, we shall omit $\Gamma$.

\smallskip
The proof of Theorem~\ref{Theorem5.1} is very similar to the one presented in \cite[Section 7.2]{Wa} in the case $G = \Z$. 

\begin{lemma}\label{Wa7.7}
Let $G$ be a semigroup, $\Gamma=(N_n)_{n\in\Z_+}$ a regular system in $G$, and $T\colon G\times X\to X$ a continuous action of $G$ on a compact metric space $X$.
\begin{itemize}
\item[(a)] For all $\ep > 0$ and all integers $n \geq 1$, $r_n(\ep,T) \leq s_n(\ep,T) \leq r_n(\ep/2,T).$
\item[(b)] Let $\aa\in\cov(X)$ and let $\delta > 0$ be a Lebesgue number for $\aa$.  For all integers $n \geq 1$, $N(\aa^n_\Gamma) \leq r_n\left(\frac{\delta}{2},T\right)$. 
\item[(c)] For any $\epsilon>0$ and any $\gamma\in\cov(X)$ with  $\diam(V) \leq \epsilon$ for all $V \in \gamma$, for all integers $n\geq 1$, $s_n(\epsilon, T) \leq N(\gamma^n_\Gamma).$
\end{itemize}
\end{lemma}
\begin{proof}
(a) corresponds to \cite[Chapter 7, Remark (5)]{Wa}, (b) and (c) to \cite[Theorem 7.7(i) and (ii)]{Wa}.
\end{proof}

The following result corresponds to \cite[Theorem 7.6]{Wa}.

\begin{lemma}\label{Wa7.6}
Let $G$ be a semigroup, $\Gamma=(N_n)_{n\in\Z_+}$ a regular system in $G$, and $T\colon G\times X\to X$ a continuous action of $G$ on a compact metric space $X$.
Let $(\aa_n)_{n\in\Z_+}$ be a sequence in $\cov(X)$ with $\diam(\aa_n)\to0$. Then $\hh(T)=\lim_{n\to \infty}\hh(T,\aa_n)$.
\end{lemma}
\begin{proof}
Suppose that $\hh(T)$ is finite, let $\epsilon>0$ and let $\gamma\in\cov(X)$ with $\hh(T,\gamma)>\hh(T)-\epsilon$. If $\delta>0$ is a Lebesgue number for $\gamma$, there exists $N\in\Z_+$ such that $\diam(\aa_N)<\delta$, and so $\gamma\prec\aa_N$. Therefore, $\hh(T,\gamma)\leq \hh(T,\aa_N)$, and hence $\hh(T)-\epsilon<\hh(T,\aa_n)\leq\hh(T)$ for every integer $n\geq N$. This shows that $\hh(T)=\lim_{n\to \infty}\hh(T,\aa_n)$.

In case $\hh(T)=\infty$, let $a>0$ and let $\gamma\in\cov(X)$ with $\hh(T,\gamma)>a$. Then proceed as in the preceding case. 
\end{proof}

\begin{theorem}\label{Theorem5.1} 
Let $G$ be a semigroup, $\Gamma=(N_n)_{n\in\Z_+}$ a regular system in $G$, and $T\colon G\times X\to X$ a continuous action of $G$ on a compact metric space $X$. Then $\hh(T) = \th(T)$.
\end{theorem}
\begin{proof}
For any integer $m \geq 1$ fix $\aa_m\in\cov(X)$ consisting of open balls of radius $2/m$ and $\gamma_m\in\cov(X)$ consisting of open balls of radius $1/ (2m)$. By Lemma~\ref{Wa7.7} we obtain, for all integers $n \geq 1$,
$$\frac{1}{n} \log N((\aa_m)_\Gamma^n) \leq \frac{1}{n} \log r_n\left(\frac{1}{m}, T\right) \leq \frac{1}{n} \log s_n\left(\frac{1}{m}, T\right) \leq \frac{1}{n} \log N((\gamma_m)_\Gamma^n).$$ 
Taking $\limsup_{n \to \infty}$, gives 
$$\hh(T, \aa_m) \leq \tr\left(\frac{1}{m}, T\right) \leq \ts\left(\frac{1}{m} , T\right) \leq \hh(T, \gamma_m) .$$
By Lemma~\ref{Wa7.6} this implies $\hh(T) \leq \th(T) \leq \hh(T)$, so $\hh(T) = \th(T)$.
\end{proof}

\subsection{Receptive topological entropy following Bowen's definition}

Next, we give another equivalent definition of the receptive topological entropy similar to the one by Bowen in \cite{Bow2}. Although this definition works for non-compact spaces as well, here we concentrate on the case of a compact metric space $X$.  Denote by $\mathcal B(X)$ the family of Borel subsets of $X$ and by  $[X]^{\leq \omega}$ {the family of all non-empty finite or countable subsets of} $X$.
 
Let $G$ be a semigroup, $\Gamma=(N_n)_{n\in\Z_+}$ a regular system in $G$ and $T\colon G\times X\to X$ a continuous action. Let also $\aa = \{ A_1, A_2, \ldots, A_k\}\in\cov(X)$. Modifying a definition in \cite{Bow2}, for every non-empty subset $E$ of $X$ we denote
$$n_{\aa}(E)= n_{T, \aa}(E) =\max\{n \in\Z_+\cup\{\infty\}\colon \aa \prec \{ gE\colon g \in N_n\}\}.$$
We set $n_{\aa}(E) = 0$ if $E$ is not contained in any element of $\aa$. Then set
$$\dd_{\aa} (E) = e^{-n_{\aa}(E)}\in\{r\in\R\colon r\leq 1\};$$
in \cite{Mi2} the notation $\diam_{\aa,T}(E) = \dd_{\aa} (E)$ was used. If $\mathcal E$ is a family of subsets of $X$, let 
$$\diam_{\aa,T}(\mathcal E ) = \sup_{E \in \mathcal E}  \dd_{\aa} (E)= \sup_{E \in \mathcal E} e^{-n_{\aa}(E)}.$$
Next, given a subset $Y$ of $X$, $\epsilon > 0$ and any $\lambda \in \R$, consider the family $\mathcal F_\epsilon(Y)$ of all finite or countable covers $\mathcal E$ of $Y$ with $\diam_{\aa,T}(\ee)  < \epsilon$. Clearly, 
$$\mathcal F_{\epsilon'}(Y)\subseteq \mathcal F_\epsilon(Y),\mbox{ whenever } \epsilon' < \epsilon.\eqno(*)$$

For a family $\ee\subseteq [X]^{\leq \omega}$ and any real $\lambda>0$ define
$$\dd_{\aa} (\ee, \lambda) = \sum_{E\in\mathcal E} \dd_{\aa} (E)^\lambda = \sum_{E\in\mathcal E} e^{- \lambda\, n_{\aa}(E)}.$$
%
Therefore, the function
\begin{equation}\label{RlY}
R_{\lambda,Y}(\epsilon) :=\inf \, \{ \dd_{\aa}(\ee, \lambda) : \ee \in \mathcal F_\epsilon(Y)\}
\end{equation}
increases when  $\ep> 0$ decreases. This allows us to define an outer measure $\m_{\aa, \lambda}$ on $X$ by setting
$$\m_{\aa,\lambda}(Y) = \lim_{\ep\to 0}R_{\lambda,Y}(\epsilon)\in\R_{\geq0}\cup\{\infty\}.$$ 

The function $\m_{\aa, \lambda}(Y)$ is non-increasing in $\lambda$ and if $0<\m_{\aa, \lambda}(Y)< \infty$  for some $\lambda>0$, then $\m_{\aa,\lambda'}(Y) = 0$ for $\lambda' > \lambda$ and $\m_{\aa,\lambda'}(Y) = \infty$ for $\lambda' < \lambda$. Set 
$$b_{\aa}(T, Y,\Gamma) = \inf \{ \lambda\in\R_{\geq0}\colon\m_{\aa,\lambda}(Y) = 0\}$$
and
$$b(T,Y,\Gamma) = \sup\{ b_{\aa}(T,Y,\Gamma)\colon \aa\in\cov(X),\ \aa\ \text{finite}\}.$$
Finally, set $b(T,\Gamma) = b(T,X,\Gamma) .$
When the system $\Gamma$ is clear from the context and no confusion is possible, we shall omit $\Gamma$.

\begin{remark}\label{lambda<0} 
Only real numbers $\lambda\geq0$ were considered above, as the negative values of $\lambda$ are irrelevant 
for the definition of $b(T,Y,\Gamma)$.  
Indeed, $\m_{\aa, \lambda}(Y)$ is non-increasing in $\lambda$ and then we take an $\inf$ in the definition of $b_\aa(T,Y,\Gamma)$. 
Actually, 
\begin{equation}\label{*eq}\m_{\aa, \lambda}(Y)=\infty \ \text{if}\ \ \lambda\leq 0.\end{equation}
Clearly, $\m_{\aa,0}(Y)=1$ if $Y$ is a singleton.
So it is enough to verify that $\m_{\aa,0}(Y)=\infty$ if $Y$ is not a singleton.
 Obviously $\dd_{\aa} (\ee, 0) = |\ee|$, so   $R_{0,Y}(\epsilon)=\inf \{ \dd_{\aa}(\ee, 0)\colon \ee \in \mathcal F_\epsilon(Y)\}= \inf \{|\ee|\colon\ee \in \mathcal F_\epsilon(Y)\} $. This implies $\m_{\aa,0}(Y)=\infty$, since $Y$ is not a singleton. 

Here is an alterative direct proof of \eqref{*eq} that works also for singleton $Y$ when  $\lambda<0$.
Let $\epsilon>0$ and $\mathcal E\in\mathcal F_\varepsilon(Y)$. Then for every $E\in\mathcal E$, $e^{-n_\aa(E)}<\epsilon$, so
$e^{-\lambda n_\aa(E)}\geq {\ep^{\lambda}},$ so $\mathcal D_\aa(\mathcal E,\lambda)=\sum_{E\in\mathcal E}e^{-\lambda n_\aa(E)}\geq \ep^{\lambda}$
and $R_{\lambda,Y}(\epsilon)\geq \ep^{\lambda}$. Therefore,  $\m_{\aa,\lambda}(Y) = \lim_{\ep\to 0}R_{\lambda,Y}(\epsilon)\geq  \lim_{\ep\to 0} \ep^{\lambda}=\infty$.
\end{remark}

%

In the next theorem the inequality $b(T)\leq\th(T)$ holds for any regular system and also when the semigroup is not commutative.
We follow the line of \cite{Bow2}, however some significant modifications will be necessary.

\begin{theorem}\label{Theorem5.3}
Let $G$ be a  commutative finitely generated semigroup with a standard regular system $\Gamma=(N_n)_{n\in\Z_+}$ and let $T\colon G\times X\to X$ be a continuous action of $G$ on a compact metric space $X$. Then $b(T)=\th(T)$.
\end{theorem}
\begin{proof} 
In view of Theorem~\ref{Theorem5.1} we show that $b(T)=\hh(T)$. First, we show that $b(T) \leq \hh(T)$. 
Let $\aa = \{ A_1, A_2, \ldots, A_k\}\in\cov(X)$, let $n \geq 1$ be an integer, and let  $\ee_n = \{ E_1, E_2 \ldots, E_m\}$ be a subcover of $\aa_\Gamma^n$ with $m = |\ee_n| = N(\aa_\Gamma^n)$. By the definition of $\aa_\Gamma^n$, any $E_i$ can be  written in the form $E_i =\bigcap_{g \in N_n}g^{-1}A_g$ for some $A_g \in \aa$. So, for any $g \in N_n$ we get $g E_i \subseteq A_g$. Hence for every $E_i$ we have $n_{\aa}(E_i) \geq n$. Therefore, given $\lambda\in\R_{\geq0}$,
$$\dd_{\aa}(\ee_n, \lambda)= \sum_{i=1}^m e^{-\lambda \, n_{\aa}(E_i)} \leq m\, e^{-\lambda n} = N(\aa^n_\Gamma)\, e^{-\lambda n}= e^{-\lambda n + \log N(\aa^n_\Gamma)} =  e^{n (-\lambda  + \frac{1}{n} \log N(\aa^n_\Gamma))}.$$
If $\lambda > \hh(T, \aa) = \limsup_{n\to\infty} \frac{1}{n} \log N(\aa^n_\Gamma)$,  then taking $\delta > 0$ so that $\lambda > \hh(T, \aa) + \delta$, we get 
$\dd_{\aa} (\ee_n, \lambda) < e^{-n \delta}$ for all sufficiently large $n\geq1$, so we must have $\m_{\aa, \lambda}(X) = 0$. This demonstrates that 
$b_{\aa}(T,X) \leq \hh(T, \aa)\leq\hh(T)$. Hence $b(T)=b (T,X)  \leq \hh(T)$.

\ms
Next, we show that $\hh(T) \leq b(T)$. The proof is articulated in four claims. 

Assume that $b(T) < \infty$; otherwise there is nothing to prove. Fix a small constant $\epsilon > 0$ and an arbitrary $\lambda\in\R_{\geq0}$ with
$$b(T) < \lambda <  b(T) + \epsilon .$$
The definition of $b(T)$ now implies $b_{\bb} (T,X) < \lambda$ and so $\m_{\bb, \lambda}(X) = 0$ for any finite $\bb\in\cov(X)$. Our aim will be to deduce from this that $\hh(T) \leq b(T) + 2 \epsilon $. Since $\epsilon $ was chosen arbitrarily small, this will  imply that  $\hh(T) \leq b(T)$. 

Let us fix a finite set of generators  $f_1,f_2,\ldots, f_\ell$ of $G$, so that $\Gamma=(N_n)_{n\in\Z_+}$ is  the standard regular system defined by $N_1=\{0,f_1,f_2,\ldots,f_\ell\}$. 
Our next step is to fix $\aa = \{ A_1, A_2, \ldots, A_k\}\in\cov(X)$ with 
\begin{equation}\label{5.3}
\hh(T, \aa) = \limsup_{n\to\infty} \frac{1}{n} \log N(\aa^n_\Gamma) > \hh(T) - \epsilon . 
\end{equation}
Let $\delta > 0$ be a Lebesgue number for $\aa$. Since $N_1$ is finite and $f_i\colon X \to X$ is uniformly continuous for every $i\in\{1,\ldots,\ell\}$, there exists $\delta _1 \in (0, \delta/2)$ such that $g  B_{\delta_1}(x) \subseteq B_\delta(g x)$ {for every} $g \in N_1$ {and every} $x\in X$.
That is, for any $g\in N_1$ and any $x, y  \in X$, $d(x,y) < \delta_1$ implies $d(gx, gy) < \delta$, so $\aa  \prec \{gB_{\delta_1}(x)\colon g\in N_1, x\in X\}$. There exists a finite cover $\bb = \{ B_1,B_2, \ldots, B_p\}$ 
of $X$ consisting of open balls of radius $\delta_1$. By what we said above, 
$$\aa \prec \{gB_i: g\in N_1, i\in\{1,2, \ldots, p\}\}.$$
Let $\delta_2 \in (0, \delta_1)$ be a Lebesgue number for $\bb$. Hence, $\bb \prec \{E\}$ for every subset $E$ of $X$ with $\diam(E) < \delta_2$.

As mentioned above, the choice of $\lambda$ implies $\m_{\bb, \lambda} (X) = 0$, so in particular there exists a finite or countable cover 
$\ee' = \{E_i' : i\in I\}$ of $X$ with $\diam(E_i') < \delta_2$ for all $i\in I$ and $\dd_{\bb}(\ee', \lambda) < 1$.  Next, for every $i\in I$ let $n_i:= n_\bb(E_i')$ and for every 
$g \in N_{n_i}$ fix a member $B_g \in \bb$, such that $gE_i' \subseteq B_g$ (we noted above that $\bb\prec\{gE_i'\}$).
Now put, for every $i\in I$, 
$$E_i:= B_{\delta_2 -\diam(E'_i)}(E'_i) \ \cap \ \bigcap_{g\in N_{n_i}}g^{-1} B_g$$
(by definition $B_\gamma(M) = \bigcup_{z\in M} B_\gamma(z)$ for every $\gamma > 0$ and every subset $M$ of $X$).
For each $i\in I$, this open set $E_i$ is ``slightly enlarging"  the set $E_i'$ to an open set with diameter $< 2\delta_2$; then $n_\bb(E_i) \leq n_\bb(E'_i) = n_i$. The definition of $E_i$ implies $n_\bb(E_i) = n_i = n_\bb(E'_i)$: indeed, for every $g\in N_{n_i}$ we have $E_i \subseteq g^{-1}B_g$, so $gE_i \subseteq B_g$, and hence $\dd_{\bb}(\ee_1, \lambda)=\dd_{\bb}(\ee', \lambda)<1$,  where $\ee_1 = \{E_i\colon i\in I\}\in\cov(X)$. 

By choosing a finite subcover $\ee = \{V_1,V_2, \ldots, V_m\}$ of $\ee_1$ we get a finite open cover of $X$ with $\dd_{\bb}(\ee, \lambda) \leq \dd_{\bb}(\ee_1, \lambda) < 1$.  Let 
$$J =  \{ 1,2, \ldots, m\}, \quad  a_j:= n_{\bb}(V_j)\ \text{for}\ i \in  J \quad \text{and} \quad M = \max_{1 \leq j \leq m} a_j.$$
Note that $a_j \geq 1$ for all $j \in J$. This implies
\begin{align*}\label{5.8}
\sum_{s=1}^\infty \; \sum_{j_1, \ldots, j_s = 1}^m e^{-\lambda(a_{j_1} + \ldots + a_{j_s}) }
= &  \sum_{s=1}^\infty \: \left(\sum_{j_1= 1}^m e^{-\lambda\; a_{j_1}}\right) \;  \left(\sum_{j_2= 1}^m e^{-\lambda\; a_{j_2}}\right) \cdots
\left(\sum_{j_s= 1}^m e^{-\lambda\; a_{j_s}}\right) \nonumber=\\
 = & \sum_{s=1}^\infty  \left(\sum_{i=1}^m  e^{-\lambda a_i}  \right) ^s= \sum_{s=1}^\infty (\dd_{\bb}(\ee, \lambda))^s < \infty
\end{align*}
(the last series converges as $\dd_{\bb}(\ee, \lambda) < 1$). 

For the sake of brevity {set $f = f_\ell$ } and let $C(V_{j}):= V_{j}$ for every $j\in J$. For $s>1$ and $j_1,\ldots,j_s\in J$, let 
\begin{align*}
C(V_{j_1}, \ldots, V_{j_s}) &=\left\{ x \in V_{j_1} : f^{a_{j_1}  +\ldots + a_{j_{r-1}}  } (x) \in V_{j_r}, \:  \forall r\in\{2, \ldots,s\} \right\}\\
& = V_{j_1} \cap f^{-a_{j_1}}  V_{j_2} \cap f^{-(a_{j_1}    +a_{j_2} ) } V_{j_3} \cap \ldots \cap   f^{-(a_{j_1}  +\ldots + a_{j_{s-1}})}V_{j_s};
\end{align*}
clearly $C(V_{j_1}, \ldots, V_{j_s})$ is an open subset of $X$. 

For any pair of integers $s\geq 1$ and $n \geq 1$  let 
$$\mathcal J_{n,s} : = \{\bar \jmath = (j_1, \ldots, j_s) \in J^s\colon n\leq a_{j_1} + \ldots   + a_{j_s} \leq  n+M \}.$$
As $a_j\geq 1$ for all $j\in  J$, it is clear that $\mathcal J_{1,1} = J$, while $\mathcal J_{n,s}\ne \emptyset $ implies $s \leq n + M$, hence 
$\sum_{s\geq 2}\sum_{\bar \jmath \in \mathcal J_{n,s} }1 <\infty$ for all $n\geq1$.  

\begin{claim} 
For any integer $n \geq 2M$, the family
$$\cc_n = \{ C(V_{j_1}, \ldots, V_{j_s}) : s\geq 2 \: ,\: \bar \jmath = (j_1, \ldots, j_s)\in\mathcal J_{n,s}  \}$$%
is an open cover of $X$.
\end{claim}
\begin{proof}
Fix $n \geq 1$, and pick $x \in X$. Define recursively indices $j_1, j_2, \ldots$ such that:
\begin{itemize}
\item[(1)] $x \in V_{j_1}$; 
\item[(2)]  if $s>1$ and $j_1, j_2, \ldots, j_{s-1}$ are already defined, pick $j_s$ such that $f^{ a_{j_1}+\ldots + a_{ j_{s-1} }} (x)\in V_{j_s}$.  
\end{itemize}
Obviously, $x\in C(V_{j_1}, \ldots, V_{j_s})$ for every integer $s\geq1$. Moreover, since $a_j\geq1$ for all $j\in I$, the sequence 
$$a_{j_1}< a_{j_1}+ a_{j_2}< \ldots <  a_{j_1}+ \ldots + a_{ j_{s} }< \ldots $$
is strictly increasing and the gaps between two adjacent members are $\leq M$. As $n \geq 2M>a_{j_1} $, there exists an $s\geq 2$ such that 
$$a_{j_1}+ \ldots + a_{ j_{s-1} }< n \leq a_{j_1}+ \ldots + a_{ j_{s-1} } + a_{ j_{s} },$$
then obviously, $a_{j_1}+ \ldots + a_{ j_{s-1} } + a_{ j_{s} }\leq n + M$, as $a_{ j_{s} }\leq M$. Hence, 
$C(V_{j_1}, \ldots, V_{j_r}) \in \cc_n$. This ends up the proof, since $x\in C(V_{j_1}, \ldots, V_{j_s})$, as mentioned above. 
\end{proof}

Recalling that $f = f_\ell$, for any integer $n \geq 1$ set
$$\aa^{n,\ell} = \bigvee_{r=0}^{n-1} f_\ell^{-r}\aa = \left\{ \bigcap_{r=0}^{n-1} f_\ell^{-r} A_{i(r)} : i : \{0,1, \ldots, n-1\} \to \{0,1,\ldots,k\} \right\} .$$

\begin{claim}\label{Claim 2} 
For all $n \geq 2M$, $\aa^{n,\ell} \prec \cc_n$.
\end{claim}
\begin{proof}
It is enough to prove the following statement for all integers $s\geq1$:
\medskip

$P(s):$ If a set of the form $C = C(V_{j_1}, \ldots, V_{j_s}) $ belongs to $\cc_n$ for an integer $n \geq 2M$, then $\aa^{n,\ell}\prec \{C\}$.

\medskip
We argue by induction on $s\geq1$. To prove $P(1)$, let $C = C(V_{j}) = V_j\in \cc_n$ for some integer $n \geq 1$. Then  $n\leq a_{j}$. So $\aa^{n,\ell} \prec \aa^{a_{j}, \ell}$, 
hence it suffices to check that $\aa^{a_{j}, \ell} \prec \{C\}$, or equivalently,  $gC$ is contained in some element of $\aa$ for every $g = f_\ell^p$, with $0 \leq p < a_{j}$. 
To this end note that  $C = V_{j}$ and $\diam(V_{j}) < \delta_2$ implies that $C$ is contained in some element $B_i$ of $\bb$. Now $a_{j} = n_{\bb}(V_{j})$ 
means that 
$gC$ is contained in some element of $\bb$, and therefore in some element of $\aa$ (since $\aa \prec \bb$) for all $g \in N_{a_{j}}$. 
Thus, $f_\ell^p C$ is contained in some element of $\aa$.

Assume that $s > 1$  and the statement $P(r)$ is true for all $1\leq r \leq s-1$. We will prove $P(s)$. Let $C = C(V_{j_1}, \ldots, V_{j_s})$ be such that $C \in \cc_n$ for some $n \geq 1$.  
Then $n\leq  a_{j_1}  + \ldots + a_{j_{s}} \leq n +M$. Set $u = a_{j_1}$, $n' = n-u$  and
$C' = C(V_{j_2}, \ldots, V_{j_s})$.
It is easy to see that 
\begin{equation}\label{LastLabel}
f_\ell^u C\subseteq C'.
\end{equation}
Since $n \leq u + (a_{j_2} + \ldots + a_{j_s}) \leq n+M$ implies that $n' \leq a_{j_2} + \ldots + a_{j_s} \leq n'+M$ and consequently $C' \in \cc_{n'}$, 
our inductive assumption gives $\aa^{n', \ell} \prec \{C'\}$. Now we show that $\aa^{n,\ell} \prec \{ C\}$, i.e., $gC$ is contained in some $A_i$ for every 
$g = f_\ell^p$ with $0\leq p < n$. The case $p < u$ follows from the argument proving $P(1)$ above, so assume that $g = f_\ell^p $  for some integer $p$ with   $u \leq p < n-1$. Then $q = p-u  < (n-1) - u =  n'-1$, so $\aa^{n', \ell} \prec\{ C'\}$ implies that $f_\ell^{q} C' \subseteq A_i$ for some $i\in\{1,\ldots,k\}$. 
 From (\ref{LastLabel}) we deduce that  $f_\ell^pC = f_\ell^q f_\ell^u C \subseteq f_\ell^q C' \subseteq A_i$.
This proves the claim. 
\end{proof}

Now we estimate $N(\aa^n_\Gamma)$. To this end, for $n,s\geq  1$ define the set 
$$\Xi_{n,s}:=\{\xi=(r_1,\ldots,r_{\ell-1};j_1,\ldots,j_s)\in\Z_+^{\ell-1}\times J^s\colon 0 \leq r_i<n,\: \forall i\in\{1,\ldots,\ell-1\},\; r_1+ \ldots + r_{\ell-1} +  a_{j_1} + \ldots   + a_{j_s} \leq  n+M \}.$$
As $a_i \geq 1$ for all $i \in  J$, it is clear that $\Xi_{n,s}\ne \emptyset $ implies $s \leq n + M$, hence $\sum_{s=1}^\infty|\Xi_{n,s}| <\infty$ for all $n\geq1$.  

\begin{claim}\label{New:Claim} 
For all $n \geq 2M$, $N(\aa^n_\Gamma)\leq \sum_{s=1}^\infty \; |\Xi_{n,s}|$.
\end{claim}
\begin{proof}
By definition
\begin{align*}
\aa^n_\Gamma & = \bigvee_{r_1 +  \ldots + r_\ell < n\atop{r_1\cdots r_\ell \geq 0}} f_1^{-r_1}\cdots f_\ell^{-r_\ell} \aa =  \bigvee_{r_1 +  \ldots + r_{\ell-1} \leq p \atop{r_1\cdots r_{\ell-1} \geq 0}\, ,0 \leq p < n} f_1^{-r_1}\cdots f_{\ell-1}^{-r_{\ell-1}} \left( \bigvee_{r=0}^{n-p-1} f_\ell^{-r}\aa \right) =\\
& =  \bigvee_{r_1 + \ldots + r_{\ell-1} \leq p \atop{r_1 \cdots r_{\ell-1} \geq 0}\, ,0 \leq p < n} f_1^{-r_1}\cdots f_{\ell-1}^{-r_{\ell-1}} \left( \aa^{n-p , \ell} \right) .
\end{align*}
Using Claim~\ref{Claim 2}, the above gives
\begin{align*}
\aa^n_\Gamma & \prec  \bigvee_{r_1 + \ldots + r_{\ell-1} \leq p \atop{r_1 \cdots r_{\ell-1} \geq 0}\, ,0 \leq p < n} f_1^{-r_1}\cdots f_{\ell-1}^{-r_{\ell-1}} \left(\cc_{n-p} \right)=\bigvee_{r_1+\ldots+r_{\ell-1} + q \leq n \atop{r_1\cdots r_{\ell-1} \geq 0, q > 0 }} f_1^{-r_1}\cdots f_{\ell-1}^{-r_{\ell-1}} \left( \cc_{q} \right)=\\
& =  \bigvee_{r_1 + \ldots + r_{\ell-1} + q \leq n \atop{r_1\cdots r_{\ell-1}\geq0, q > 0 }} f_1^{-r_1}\cdots f_{\ell-1}^{-r_{\ell-1}} \left(\{C(V_{j_1}, \ldots,V_{j_s}) \colon s\geq 1,\: (j_1,\ldots,j_s)\in J^s,\:  q \leq a_{j_1} + \ldots   + a_{j_s} \leq  q + M \} \right).
\end{align*}
Therefore
\begin{align*}
N(\aa^n_\Gamma)& \leq | \{ f_1^{-r_1}\cdots f_{\ell-1}^{-r_{\ell-1}} C(V_{j_1}, \ldots, V_{j_s})\colon s\geq 1,\: (j_1,\ldots,j_s)\in J^s,\:  r_1\cdots r_{\ell-1} \geq 0,\: q > 0,\\
& \qquad r_1 + \ldots + r_{\ell-1} + q \leq  n\:, \: q \leq a_{j_1} + \ldots   + a_{j_s} \leq  q + M \} | .
\end{align*}
Notice that the inequalities $r_1 + \ldots + r_{\ell-1} + q \leq n$, $q > 0$ and  $a_{j_1} + \ldots   + a_{j_s} \leq  q + M$ imply
$$r_1 + \ldots + r_{\ell-1} +  a_{j_1} + \ldots   + a_{j_s} \leq (n-q) + (q+M) =  n + M .$$
Thus, 
\begin{align*}
N(\aa^n_\Gamma) & \leq  \sum_{s=1}^\infty | \{ f_1^{-r_1}\cdots f_{\ell-1}^{-r_{\ell-1}} C(V_{j_1}, \ldots, V_{j_s})\colon \\ &\quad \quad \quad(j_1,\ldots,j_s)\in J^s,\: 0 \leq r_1, \cdots ,r_{\ell-1}  < n, \; 
r_1 + \ldots + r_{\ell-1} +  a_{j_1} + \ldots   + a_{j_s} \leq  n + M \} |\\
& \leq \sum_{s=1}^\infty \; |\Xi_{n,s}| ,
\end{align*} 
This proves Claim~\ref{New:Claim}.
\end{proof}

\begin{claim}\label{Claim 4}   
For all $n \geq 2M$ and $C =  \frac{1}{(1- e^{-\lambda})^{\ell-1}} $, 
$$\sum_{s=1}^\infty \; |\Xi_{n,s}| \leq  C\; e^{\lambda (n+M)} \sum_{s=1}^\infty (\dd_{\bb}(\ee, \lambda))^s .$$
\end{claim}
\begin{proof} 
Since $\lambda \geq 0$, and $\xi \in \Xi_{n,s}$ for some integer $s\geq1$ implies $r_1+ \ldots + r_{\ell-1} +  a_{j_1} + \ldots   + a_{j_s} \leq  n+M$, it follows that
$$1 \leq e^{\lambda[n+M - (r_1+ \ldots + r_{\ell-1} + a_{j_1} + \ldots   + a_{j_s} )]} = e^{-\lambda (r_1+ \ldots + r_{\ell-1} + a_{j_1} + \ldots   + a_{j_s} ) + \lambda (n+M)} =  e^{-\lambda (r_1+ \ldots + r_{\ell-1} + a_{j_1} + \ldots   + a_{j_s} )}  e^{\lambda (n+M)} .$$
Therefore
\begin{equation}\label{5.12}
\sum_{s=1}^\infty |\Xi_{n,s}| = \sum_{s=1}^\infty  \sum_{\xi \in \Xi_{n,s} }1 \leq
\left(\sum_{s=1}^\infty\:\sum_{\xi \in  \Xi_{n,s} } e^{-\lambda (r_1 + \ldots + r_{\ell-1} + a_{j_1} + \ldots   + a_{j_s} )} \right) e^{\lambda (n+M)}.
\end{equation}
Using the fact that $0 \leq r_i <n$ for all $i\in\{1,\ldots,\ell-1\}$, we deduce
\begin{align*}
\sum_{\xi \in  \Xi_{n,s} } e^{-\lambda (r_1 + \ldots + r_{\ell-1} + a_{j_1} + \ldots   + a_{j_s} )}& \leq \sum_{r_1=0}^{n-1} \sum_{r_2=0}^{n-1} \ldots \sum_{r_{\ell-1}=0}^{n-1} \sum_{a_{j_1} + \ldots + a_{j_s} \leq n+M}(e^{-\lambda \,r_1} \ldots e^{-\lambda \,r_{\ell-1}})\cdot e^{-\lambda (a_{j_1} + \ldots   + a_{j_s} )} \\
& \leq \left(\sum_{r=0}^{n-1} e^{-\lambda\, r}\right)^{\ell-1} \; \sum_{j_1, \ldots, j_s = 1}^m e^{-\lambda(a_{j_1} + \ldots + a_{j_s}) } \\
&\leq \frac{1}{(1- e^{-\lambda})^{\ell-1}} \;   \left(\sum_{i=1}^m  e^{-\lambda a_i}\right)^s   = C\, \left(\dd_{\bb}(\ee, \lambda)\right)^s  . 
\end{align*}
Combining this with \eqref{5.12} proves the required inequality. 
\end{proof}

Now we can conclude the proof of the theorem by using Claim \ref{Claim 4} and Claim~\ref{New:Claim}. For every $n\geq 2M$,
 \begin{align*}
 e^{- \lambda n} N(\aa^n_\Gamma) & = e^{M \lambda} e^{- \lambda (n+M)}  N(\aa^n_\Gamma) \leq e^{M \lambda} e^{- \lambda (n+M)} \; \sum_{s\geq 1}  |\Xi_{n,s}|\leq \\
& \leq  e^{M \lambda}   e^{-\lambda (n+M)} \; C\; e^{\lambda (n+M)} \sum_{s=1}^\infty (\dd_{\bb}(\ee, \lambda))^s =  C\, e^{M \lambda} \; \sum_{s=1}^\infty (\dd_{\bb}( \ee, \lambda))^s =: K 
\end{align*}
for some constant $K < \infty$ independent of $n$. This yields, for every $n\geq 2M$,
$$\frac{1}{n} \log N(\aa^n_\Gamma) \leq \lambda + \frac{1}{n} \log K.$$
Therefore 
$$\hh(T, \aa) = \limsup_{n\to \infty} \frac{1}{n} \log N(\aa^n_\Gamma) \leq \lambda < b(T) + \epsilon.$$
Combining this with \eqref{5.3} gives $\hh(T) < b(T) + 2\epsilon$. Letting $\epsilon \to 0$ implies $\hh(T) \leq b(T)$. This proves the theorem.
\end{proof}

\begin{remark}\label{Remark5.4}
Notice that the first part of the proof of Theorem~\ref{Theorem5.3} works without assuming that the semigroup is commutative and finitely generated, and moreover without any  specific assumptions about the regular system. Thus, $b(T) \leq \th(T)$ holds for any continuous action of a semigroup $G$ on a metric space and with respect to any regular system in $G$.
\end{remark}

\subsection{An alternative definition following Pesin}

Let $G$ be a semigroup, $\Gamma=(N_n)_{n\in\Z_+}$ a regular system in $G$, and $T\colon G\times X\to X$ a continuous action of $G$ on a compact metric space $X$.
In what follows we will give an equivalent definition of Bowen's entropy $b(T,Y)$ following some ideas of Pesin (see \cite[Section 4.11]{P}). 

Let $Y$ be a non-empty subset of $X$ and let $\lambda \geq 0$. Given $\epsilon > 0$ and an integer $N \geq 1$ set 
\begin{equation}\label{om}
\om_{N, \epsilon, \lambda}(Y) = \inf\left\{\sum_{i\in I} e^{-\lambda n_i}\colon |I|\leq\omega, \ Y\subseteq \bigcup_{i\in I} D_{n_i}^\Gamma(x_i, \epsilon),\ x_i \in Y,\ n_i \geq N\ \forall i\in I\right\}.
\end{equation}
Then define
$$\om_{\epsilon,\lambda}(Y) = \lim_{N\to \infty} \om_{N, \epsilon, \lambda}(Y) \quad\text{and} \quad \om_{\lambda}(Y) = \lim_{\epsilon\to 0} \om_{\epsilon, \lambda}(Y) .$$
The limits exist due to the monotonicity of $\om_{N, \epsilon,\lambda}(Y)$ with respect to both $N$ and $\epsilon$.
Just as in the case of $\m_{\aa,\lambda}(Y)$, it is easy to see that there exists a critical point $s \in \R$ such that
$\om_{\lambda}(Y) = 0$ for $\lambda  > s$ and $\om_{\lambda}(Y) = \infty$ for $\lambda < s$. Set $c (T,Y,\Gamma) = s$. Thus, we define a new entropy-like quantity
$$c (T,Y,\Gamma) = \inf \{ \lambda \geq 0\colon \om_{\lambda}(Y) = 0 \} .$$
For $Y=X$ set $c(T,\Gamma) = c(T,X,\Gamma)$.
When the system $\Gamma$ is clear from the context and no confusion is possible, we shall omit $\Gamma$.

\begin{remark}
In \cite{B-S2}, following  the line of \cite[Sections 10 and 11]{P}, we defined the upper capacity for Pesin-Carath\'eodory structures, and we described the receptive topological entropy $\th(T,Y,\Gamma)$ as a limit of upper capacities for suitable Pesin-Carath\'eodory structures. The entropy $c(T,Y,\Gamma)$ can be analogously obtained as a limit of upper capacities.
\end{remark}


\begin{theorem}\label{Theorem5.5} 
Let $G$ be a semigroup, $\Gamma=(N_n)_{n\in\Z_+}$ a regular system in $G$, and let $T\colon G \times X \to X$ be a continuous action of $G$ on a compact metric space $X$. If $Y$ is a non-empty subset of $X$, then $b(T, Y,\Gamma) = c(T,Y,\Gamma)$. 
\end{theorem}
\begin{proof}
Let $\aa = \{ A_1, \ldots, A_k\}\in\cov(X)$ and let $\epsilon > 0$ be a Lebesgue number for $\aa$. 
Given an integer $N \geq 1$ and a real number $\lambda>0$, set
\begin{equation}\label{5.13}
\m_{N,\aa,\lambda}(Y) = \inf \, \left\{ \sum_{E\in\mathcal E}e^{- \lambda \, n_{\aa}(E)}\colon \ee\subseteq [\mathcal B(X)]^{\leq \omega},\ Y \subseteq \bigcup \ee,\ n_{\aa}(E) \geq N \; \forall \;  E\in\ee \right\} .
\end{equation}
The map $N \to \m_{N,\aa,\lambda}(Y)$ is monotone, and it can be seen as above that $\m_{N,\aa,\lambda}(Y)$ has similar properties to $\om_{N, \epsilon, \lambda}(Y).$

Notice that if $E = D_n^\Gamma(x, \epsilon)$ for some $x\in X$ and $n \geq 1$, then $n_{\aa}(E) \geq n$. Indeed, given $g \in N_n$, there exists $j\in\{1,\ldots,k\}$ with $B_{\epsilon}(gx) \subseteq A_j$. 
For any $y\in E$ we have $d(gx,gy) < \epsilon$, so $gy \in B_{\epsilon}(gx) \subseteq A_j$. Thus, $gE \subseteq A_j$. This proves that $n_{\aa}(E) \geq n$.

Hence, taking $N \geq 1$ sufficiently large so that $e^{-N} < \epsilon$, the set in the right-hand side of \eqref{5.13} is contained in the set in the right-hand side of \eqref{RlY}. Therefore $\m_{N,\aa,\lambda}(Y) \leq \om_{N, \epsilon, \lambda}(Y)$. 
Moreover, the set in the right-hand side of \eqref{om} is contained in the set in the right-hand side of \eqref{5.13}.
Taking limits $N \to \infty$, $\epsilon \to 0$, gives $\m_{\aa, \lambda}(Y) \leq \om_{\lambda}(Y)$. Thus, whenever $\om_{\lambda}(Y)= 0$ we have
$\m_{\aa, \lambda}(Y) = 0$ as well. Hence $c(T,Y) \geq b_{\aa}(T,Y)$ which implies $c(T,Y) \geq b(T,Y)$.

\smallskip
To prove the opposite inequality, take small constants $\delta > 0$ and $\epsilon > 0$, and let $\aa = \{ A_1, , \ldots, A_k\}\in\cov(X)$ with $\diam(A_i) < \epsilon$ for all $i\in\{1,\ldots, k\}$ and such that  $\lambda_0 = b_{\aa}(T,Y) > b(T,Y) - \delta$.  
Then $\m_{\aa,\lambda}(Y) = 0$ for all $\lambda > \lambda_0$. 
Let $\lambda > \lambda_0$ and take a large $N \geq 1$ so that $e^{-N} < \epsilon$ and $\m_{N,\aa,\lambda}(Y) < \delta$. 
Then there exists a cover $\ee = \{E_i\colon i\in I\}$ of $Y$ as in the right-hand side of \eqref{5.13} with
$$\sum_{i\in I} e^{-\lambda n_{\aa}(E_i)} < \delta .$$

For each $i\in I$, fix an arbitrary $x_i \in E_i$ and set $n_i = n_{\aa}(E_i)$. 
Notice that $E_i \subseteq D_{n_i}^\Gamma(x_i,\epsilon)$: indeed, for every $g \in N_{n_i}$ there exists $j\in\{1,\ldots,k\}$ with $gE_i \subseteq A_j$; since $\diam(A_j) < \epsilon$, we get $d(gx_i,gy) < \epsilon$ for all $y \in E_i$. 
This yields $Y \subseteq\bigcup_{i\in I} D_{n_i}^\Gamma(x_i,\epsilon)$, which in turn shows that $\om_{N, \epsilon, \lambda}(Y) < \delta$. Taking limits
$N\to \infty$, $\epsilon \to 0$, gives $\om_{\lambda}(Y) < \delta$. Letting $\delta \to 0$, it now follows that $\om_{\lambda}(Y) = 0$ for all $\lambda > b(T,Y)$. Hence $c(T,Y) \leq b(T,Y)$.
\end{proof}

\def\hloc{\underline{h}^{loc}}
\def\hLoc{\overline{h}^{loc}}
\def\B{\overline{B}}
\def\L{\underline{L}}
\def\S{\underline{S}}

The following is a direct consequence of Theorem~\ref{Theorem5.3} and Theorem~\ref{Theorem5.5}.

\begin{corollary} 
Let $G$ be a finitely generated commutative semigroup, $\Gamma=(N_n)_{n\in\Z_+}$ a standard regular system in $G$, and $T\colon G\times X\to X$ a continuous action of $G$ on a compact metric space $X$. Then $\th(T,\Gamma) = b(T,\Gamma) = c(T,\Gamma)$.
\end{corollary}

\section{Comparisons between metric and topological entropy}\label{VPsec}

Let again $T\colon G\times X\to X$ be a continuous action of the semigroup $G$ on a compact metric space $X$, and let $\Gamma=(N_n)_{n\in\Z_+}$ be a regular system in $G$.

Denote by $M(X)$ the \emph{space of all Borel probability measures} on $X$ considered with the weak$^*$ topology (see e.g. \cite[Chapter 6]{Wa}). Let $M(X,T)$ be the closed \emph{subspace of $M(X)$ consisting of all $T$-invariant measures} $\mu \in M(X)$.

\begin{remark}\label{existenceofmeasure} 
In general, even when the semigroup $G$ is finitely generated, $M(X,T)$ could be empty
 (e.g., see \cite[Example 4.1.1]{Walc}). Nevertheless, it is known (see \cite[pages 97-98]{Walc}) that for two commuting homeomorphisms $f,g\colon X \to X$ there exists a Borel probability measure on $X$ which is both $f$-invariant and $g$-invariant. The argument can easily be extended to the case of two, and so finitely many, pairwise commuting {continuous selfmaps}. Therefore, $M(X,T)$ is non-empty in case $G$ is a finitely generated commutative semigroup. 
 
 The existence of a $T$-invariant Borel probability measure on $X$ is ensured also when $G$ is an amenable (not necessarily finitely generated) group (see \cite[Theorem 8.10]{EW}).
\end{remark}

The classical \emph{Variational Principle} due to Goodwyn \cite{G} and Goodman~\cite{Go1} states that, in case $G=\Z_+$ and $T=T(1,-)\colon X\to X$ is a continuous selfmap,
\begin{equation}\label{7.1}
h(T) = \sup \{ h_\mu(T) : \mu \in M(X,T) \}.
\end{equation}
In fact there is a more general Variational Principle concerning the topological pressure. Proofs of the Variational Principle have been given by various authors 
in a variety of specific situations -- see \cite{Ru1} (for $G = \Z^k$ and topological pressure, under some conditions),
\cite{El,Mi1,Oll,OllP,Ru2} and others. 

\begin{remark}\label{VPam}
In particular, it is well-known (see \cite{Oll,OllP}) that the Variational Principle \eqref{7.1} holds for continuous actions of 
amenable groups, where $h(T)$ and $h_{\mu}(T)$ are the classical topological and metric entropy defined by means of a F\o lner sequence $(N_n)_{n\in\Z_+}$ (see Section~\ref{classicalhm} and Section~\ref{classicaltop}). Moreover, the same result was proved by Misiurewicz \cite{Mi1} for actions of $\Z_+^n$, and for actions of countable cancellative semigroups in \cite{ST}.
\end{remark}

For the receptive topological entropy $\th(T)$ considered here a few remarks follow (the general case is not done yet). 
What we aim to prove is the following Receptive Variational Principle.

\begin{conjecture}\label{MainConj}
Let $G$ be a semigroup, $\Gamma=(N_n)_{n\in\Z_+}$ a regular system in $G$ and let $T\colon G\times X\to X$ be a continuous action on a compact metric space $X$ with  $M(X,T)\neq\emptyset$. 
Then:
\begin{equation}\label{7.2}
\th(T) = \sup \{\th_\mu(T)\colon \mu \in M(X,T) \}.
\end{equation}
\end{conjecture}

The main problem in all existing proofs of similar claims, including the one by Misiurewicz in \cite{Mi1}, is that in doing some estimates the ``error term''
is multiplied by $|N_n|$. Then using the classical definition of entropy, a division by $|N_n|$ follows which solves the problem.
In our case we only divide by $n$ and this cannot kill the extra factor.

\begin{proposition}\label{Remark1}
Let $G$ be a (necessarily countable) amenable cancellative semigroup with a regular system $(N_n)_{n\in\Z_+}$ which is a F\o lner sequence in $G$ with $\lim_{n\to \infty} \frac{n}{|N_n|}=0$.
Let $T\colon G\times X\to X$ be a continuous action of $G$ on a compact metric space $X$. If $h(T) > 0$, then \eqref{7.2} holds. 
\end{proposition}
\begin{proof}
In this case we must have $\th(T) = \infty$, so clearly $\th_\mu(T) \leq \th(T)$ for all $\mu\in M(X,T)$. Let us prove that $\sup\{ \th_\mu(T)\colon\mu\in M(X,T)\} = \infty$. 
Fix an arbitrary $\ep > 0$ such that $0 < \ep < h(T)$. By Remark~\ref{VPam} there exists $\mu \in M(X,T)$ such that
$\ep < h_\mu(T) \leq h(T)$. By the definition of $h_\mu(T)$, there exists a finite $\aa\in\pa(X)$ with 
$$\ep < h_\mu(T, \aa) = \lim_{n\to\infty} \frac{1}{|N_n|} H_\mu(\aa^n_\Gamma) .$$
Since $|N_n|/n\to \infty$ as $n \to \infty$, this implies 
$\th_\mu(T, \aa) = \limsup_{n\to \infty} \frac{1}{n} H_\mu(\aa^n_\Gamma) = \infty .$
Thus, $\th_\mu(T) = \infty$, so \eqref{7.2} holds.
\end{proof}

\begin{conjecture}\label{conjVP1/2}
Let $G$ be a semigroup, $\Gamma=(N_n)_{n\in\Z_+}$ a regular system in $G$ and let $T\colon G\times X\to X$ be a continuous action on a compact metric space $X$  with  $M(X,T)\neq\emptyset$. Then
\begin{equation}\label{geqVP}
\th(T) \geq \sup \{ \th_\mu(T)\colon \mu \in M(X,T) \}.
\end{equation}
\end{conjecture}

\begin{question} 
For which compact metric spaces does \eqref{geqVP} hold? More specifically, does it hold for the Cantor cube $\{0,1\}^n$? What about an arbitrary totally disconnected compact metric space $X$?  What about the tori $\T^n$, or, more generally, metric continua? 
\end{question}

We conclude with a special case, namely when $X$ is a compact topological group with Haar measure $\mu$ and  $G$ acts on $X$ by continuous surjective endomorphisms. By a well known Halmos paradigm, this ensures that  the actions $T$ is measure preserving. In case $G=\Z_+$, $\th(T) = h(T)=h_\mu(T)=\th_\mu(T)$, as established in \cite{S}. The equality $h(T)=h_\mu(T)$ is true also for $\Z^d$ actions on compact groups for the classical entropy (see \cite{LSW}). 
Clearly, this can be seen as a (strongly) positive answer to Conjecture \ref{MainConj}. In this vein one can ask: 

\begin{question} 
Under which conditions on the semigroup $G$ the equality $\th(T)=\th_\mu(T)$ is true in the above situation?
\end{question}

The notion of algebraic entropy $h_{alg}(f)$ for an endomorphism $f\colon A \to A$ of an abelian group $A$ was briefly introduced in \cite{AKMc} and further studied in \cite{DGSZ,DG}. A nice connection, named {\em Bridge Theorem}, was found between this entropy and the topological entropy $h(\widehat f)$ 
of the dual (continuous) endomorphism $\widehat f$ of the (compact) Pontryagin dual $\widehat A$, namely $h_{alg}(f)=  h(\widehat f)$ (see \cite{DG1}). This result extends earlier previous ones of M. Weiss \cite{W} (for $A$ torsion) and Peters \cite{Pet} (for $A$ countable and $f$ an automorphism).

Extending the definition of $h_{alg}$ to the case of an action $T$ of an amenable cancellative semigroup $G$ on an abelian group $A$ via endomorphisms, 
the algebraic entropy $h_{alg}(T)$ was defined in \cite{DFG}. Moreover,  the Bridge Theorem was extended in this setting as well, in case the group $A$ is torsion. Namely, if $T$ is an action of an amenable cancellative semigroup $G$ on a torsion abelian group $A$ via endomorphisms, then 
the topological entropy $h(\widehat T)$ of the dual action $\widehat T$ of $G$ on the (compact and totally disconnected) Pontryagin dual $\widehat A$ coincides with $h_{alg}(T)$. 

Recently, the authors \cite{B-S1} defined the  \emph{receptive algebraic entropy} $\th_{alg}(T)$ of an action $T$ of a semigroup $G$, provided with a regular system $\Gamma = (N_n)_{n\in\Z_+}$, on an abelian group $A$ by endomorphisms. In case the abelian group $A$ is torsion, a Bridge Theorem between this entropy $\th_{alg}(T)$ and the receptive topological entropy $\th(\widehat T)$ of the dual action $\widehat T$ was established in \cite{GB}. 

There is an obvious analogy between the Variational Principle, connecting (receptive) topological entropy and (receptive) metric entropy of the same action $T$, and the Bridge Theorem, connecting the  (receptive) topological entropy of an action $T$ and the (receptive) algebraic entropy of the dual action $\widehat T$ (for example see \cite{DG3}). This motivated our choice to formulate Conjecture \ref{MainConj} in view of the positive evidence, in this sense, provided by the Bridge Theorem available in all these cases, as well as the results of Section \ref{local} that are in the spirit of the Variational Principle.


\section{Receptive local metric entropy}\label{local} 

For a continuous selfmap $f\colon X\to X$  of a metric space $X$ and a Borel probability measure  $\mu$ on $X$, Bowen \cite{Bow2} introduced a topological entropy $h^B$ of subsets inspired by Hausdorff dimension. Then Brin and Katok \cite{BrK} introduced a notion $\hloc_\mu$ and $\hLoc_\mu$ of lower and upper local metric entropy.
By the so-called Brin-Katok Formula $h_\mu(f)=\hloc_\mu(f)=\hLoc_\mu(f)$  they gave a description of the metric entropy in terms of the local metric entropy for every $f$-invariant probability measure $\mu$ on $X$. 

Later on, Ma and Wen \cite{MaW} showed that the lower local metric entropy $\hloc_\mu$  with respect to any Borel probability measure $\mu$ on $X$ is always smaller than Bowen topological entropy. This result was taken to a local Variational Principle by Feng and Huang \cite{FH}, who proved in particular that, in case $X$ is a compact metric space, 
$$h^B(f)=h(f)=\sup\{\hloc_\mu(f)\colon \mu\ \text{Borel probability measure on $X$}\}.$$

More recently, Bowen's topological entropy inspired by Hausdorff dimension was extended to continuous actions of amenable groups on compact metric spaces by Zheng and Chen \cite{ZC}, and under suitable hypotheses they proved that this Bowen topological entropy coincides with the classical topological entropy.
Moreover, they introduced upper and lower local metric entropy, extended the Brin-Katok Formula and prove a ``local Variational Principle'' between Bowen topological entropy and the lower local metric entropy in the case of actions of countable amenable groups on compact metric spaces.

The definition of upper and lower receptive local metric entropy that we use here is the analogue of the one in \cite{ZC} adapted to our definition of receptive topological entropy. The pointwise concepts  $\hloc_{\mu}(x)$ and $\hloc_{\mu}(x,{\epsilon})$ below are as in \cite{Bis3}, where they were introduced in the particular case of actions of finitely generated groups on compact metric spaces with respect to the finite generating set of the group.

\begin{definition}\label{LocalMetricE}
Let $T\colon G\times X\to X$ be a continuous action of the semigroup $G$ on a compact metric space $X$, and let $\Gamma=(N_n)_{n\in\Z_+}$ be a regular system in $G$. Let $\mu$ be a Borel probability measure on $X$. The \emph{lower} and the \emph{upper receptive local metric entropy} of $T$ at $x\in X$ are defined by
$$\hloc_{\mu} (x) = \lim_{\epsilon \to 0} \hloc_{\mu} (x,\epsilon) \quad\text{and}\quad \hLoc_{\mu} (x) = \lim_{\epsilon \to 0} \hloc_{\mu} (x,\epsilon),$$
where, for $\epsilon>0$,
$$\hloc_{\mu} (x,\epsilon) = \liminf_{n\to \infty} - \frac{1}{n} \log \mu(D_n^\Gamma(x, \epsilon))\quad\text{and}\quad \hLoc_{\mu} (x,\epsilon) = \limsup_{n\to \infty} - \frac{1}{n} \log \mu(D^\Gamma_n(x, \epsilon)).$$

The \emph{lower} and the \emph{upper receptive local metric entropy} of $T$ with respect to $\mu$ are defined by 
 $$\hloc_{\mu}(T,\Gamma) = \int_X \hloc_{\mu}(x)d\mu\quad \text{and}\quad \hLoc_{\mu}(T) = \int_X \hLoc_{\mu}(x) \, d\mu.$$
\end{definition}

Clearly, $\hloc_{\mu} (x,\epsilon)$ and $\hLoc_{\mu}(x)$ measure the decay of the dynamic ball $D^\Gamma_n(x,\ep)$ with respect to $\mu$.
We shall omit to write $\Gamma$ in $\hloc_{\mu}(T,\Gamma)$ when it is clear from the context.

\medskip
The argument in the proof of the  following folklore fact was kindly proposed to us by Hans Weber in the case of the trivial action. It is used in the above definition to apply the integral to the function $\hloc_{\mu}\colon X\to \R_{\geq0}$.

\begin{lemma}\label{measurableh}
Let $T\colon G\times X\to X$ be a continuous action of the semigroup $G$ on a compact metric space $X$, let $\Gamma=(N_n)_{n\in\Z_+}$ be a regular system in $G$ and $\mu$ a Borel probability measure on $X$. The functions $\hloc_{\mu}(-,\epsilon)\colon X\to \R_{\geq0}$ for $\ep>0$ and $\hloc_{\mu}\colon X\to \R_{\geq0}$ are measurable.
\end{lemma}
\begin{proof} 
We prove that, fixed $n\in\Z_+$ and $\ep>0$, the function $f_n\colon X\to [0,1]$, $x\mapsto \mu(D_n^\Gamma(x, \epsilon))$ is measurable. So also the function $x\mapsto \frac{1}{n}\log \mu(D_n^\Gamma(x, \epsilon))$ is measurable.
Then $\hloc_{\mu} (-,\epsilon)$ and $\hloc_{\mu}$ result to be measurable as limits of measurable functions.

Let $f=f_n$. First we verify that $f$ is lower semicontinuous. To this end, fix $x_0\in X$ and $a\in[0,1]$ with $f(x_0)>a$. Since 
$(D_n^\Gamma(x_0,\ep-\frac{1}{k}))_{k>0}$ is an increasing chain and $\bigcup_{k>0}D_n^\Gamma(x_0,\ep-\frac{1}{k})=D_n^\Gamma(x_0,\ep)$, the sequence of real numbers $(\mu(D_n^\Gamma(x_0,\ep-\frac{1}{k})))_{k>0}$ converges to $f(x_0)=\mu(D_n^\Gamma(x_0,\ep))$. Thus, there exists $0<\ep' <\ep$ such that 
$$\mu(D_n^\Gamma(x_0,\ep'))>a.$$
Let $\bar\ep=\ep-\ep'$. For every $g\in N_n$, $g\colon X\to X$ is (uniformly) continuous, so there exists $\delta_g>0$ such that $x\in B(x_0,\delta_g)$ implies  $gx\in B(gx_0,\bar\ep)$. Let $\delta=\min_{g\in N_n} \delta_g$ and $U=B(x_0,\delta)$. If $x\in U$, then 
\begin{equation}\label{ep'}
D_n^\Gamma(x_0,\ep')\subseteq D_n^\Gamma(x,\ep).
\end{equation}
Indeed, $y\in D_n^\Gamma(x_0,\ep')$, then for every $g\in N_n$, $d(gy,gx)\leq d(gy,gx_0)+d(gx_0,gx)\leq \ep'+\bar\ep=\ep.$
By \eqref{ep'}, if $x\in U$, then $f(x)=\mu(D_n^\Gamma(x,\ep))\geq \mu(D_n^\Gamma(x_0,\ep'))>a$. This shows that $f$ is lower semicontinuous.

The fact that $f$ is lower semicontinuous implies that $f^{-1}((\gamma,1])$ is open in $X$ for every $\gamma\in(0,1)$, and this last property entails that $f$ is measurable.
\end{proof}

For reader's convenience, we give an auxiliary result based on the classical theorems of Lusin and Egorov 
adapted to our more specific situation.  

\begin{lemma}\label{L6.1}\label{L6.2}
Let $X$ be a compact metric space and $\mu$ a Borel probability measure on $X$. 
\begin{itemize}
\item[(a)] A real valued function $f$ on $X$ is measurable if and only if for every $\ep>0$ there exists a compact subset $K$ of $X$ with $\mu(X\setminus K)<\ep$ and such that $f\restriction_K$ is continuous.
\item[(b)]  If a sequence of measurable real valued functions $f_n$ on $X$ converges to a real valued measurable function $f$ on $X$ at each $x\in X$, 
then for each $\epsilon>0$ there exists a compact subset $K$ of $X$ such that $\mu(X\setminus K)<\ep$ and $f_n$ converges to $f$ uniformly on $K$.
\end{itemize}
\end{lemma}

We see that the lower receptive local metric entropy is always smaller than the receptive metric entropy. This can be seen as one half of the receptive version of the Brin-Katok Formula.

\begin{theorem}\label{Proposition6.2}
Let $G$ be a semigroup, $\Gamma=(N_n)_{n\in\Z_+}$ a regular system in $G$ and $T\colon G\times X\to X$ a continuous action of $G$ on a compact metric space $X$.  Then, for a $T$-invariant Borel probability measure $\mu$ on $X$, 
$\hloc_{\mu}(T) \leq \th_\mu(T)$.
\end{theorem}
\begin{proof}
\textsc{Case 1}: $\hloc_{\mu}(T) < \infty$. 
Let $\delta > 0$, $g= \hloc_{\mu}$ and $g_m=  \hloc_{\mu}(-,1/m)$ for every integer $m>0$. 
Pick $\epsilon > 0$ such that
$$\hloc_\mu(T) = \int_X g(y)\, d\mu \leq \int_K g(y)\, d\mu + \delta$$
whenever $K$ is a compact subset of $X$ with $\mu(X\setminus K) < \epsilon$.

Clearly, $g(x) = \lim_{m\to\infty} g_m(x)$ for all $x \in X$. Since each $g_m$ is measurable by Lemma~\ref{measurableh}, by iterated applications of Lemma~\ref{L6.1}, one can find a compact subset $K$ of $X$ such that all functions $g_m$ and $g$ are continuous on $K$, $g_m$ converges uniformly to $g$ on $K$ and $\mu(X\setminus K) < \epsilon$. Consequently, 
$$\int_X g(y)\, d\mu\leq\int_{K}  g(y)\, d\mu + \delta.$$
Take $m_0 \geq 1$ so large that $g(y) - g_m(y) < \delta$ for all $y \in K$ and all $m \geq m_0$. Then, for all $m \geq m_0$,
$$\int_K g(y)\, d\mu \leq \int_K g_m(y)\, d\mu +\delta.$$

Next, fix an arbitrary integer $m \geq m_0$. Let $\aa = \{A_1,A_2, \ldots, A_k\}$ be a $\mu$-measurable partition of $X$ such that $\diam (A_i) < 1/m$ for all $i\in\{1,2,\ldots,k\}$ (this is straightforward, otherwise see for example \cite[Lemma 8.5]{Wa} for an even stronger statement). For any  integer $N \geq 1$ set 
\begin{equation}
K_{N} = \left\{ x\in K\colon - \frac{1}{n} \log \mu(D_n^\Gamma(x, 1/m)) > \hloc_{\mu}(x,1/m) - \delta, \: \forall \: n \geq N   \right\} .
\end{equation}
This defines an increasing sequence $(K_N)_{N\geq1}$ of subsets of $K$ with $\bigcup_{N\geq 1} K_N = K$. So there exists an integer $N\geq1$ such that
$$\int_K g(y)\, d\mu\leq\int_{K_N}  g(y)\, d\mu +\delta .$$ 
Let $n \geq N$. Then
\begin{align*}
\hloc_\mu(T) & \leq   \int_K g(y)\, d\mu +  \delta \leq \int_{K_N}  g(y)\, d\mu  + 2\delta \leq \int_{K_N}  g_m(y)\, d\mu  + 3\delta\leq \\
& \leq-\frac{1}{n}\int_{K_N}\log \mu(D_n^\Gamma(y, 1/m)) \, d\mu (y) + 4\delta= - \frac{1}{n} \sum_{Y\in \aa^n_\Gamma} \int_{K_N \cap Y} \log \mu(D^\Gamma_n(y, 1/m)) \, d\mu (y) + 4\delta .
\end{align*}
If $Y \in \aa^n_\Gamma$ is such that $Y \cap K_{N} \neq \e$, then  for any $y \in Y\cap K_{N}$ we have $Y \subseteq D_n^\Gamma(y,1/m)$,
so $\mu(Y) \leq \mu(D_n^\Gamma(y,1/m)$. Combining  the above inequalities we obtain
\begin{align*}
\hloc_\mu(T) & \leq  - \frac{1}{n} \sum_{Y\in \aa^n_\Gamma} \int_{K_N \cap Y} \log \mu(D_n^\Gamma(y, 1/m)) \, d\mu (y) + 4\delta\leq \\
& \leq - \frac{1}{n} \sum_{Y\in \aa^n_\Gamma} \int_{K_N \cap Y} \log \mu(Y) \, d\mu + 4\delta \leq   - \frac{1}{n} \sum_{Y\in \aa^n_\Gamma} \mu(Y)\, \log \mu(Y)  + 4\delta .
\end{align*}
Letting $n \to \infty$ we obtain that $\hloc_{\mu}(T) \leq \th_{\mu}(T, \aa) + 4 \delta$,  therefore
$\hloc_{\mu}(T) \leq \th_{\mu}(T) + 4\delta$. This is true for any $\delta > 0,$ so $\hloc_{\mu}(T) \leq \th_{\mu}(T)$.

\smallskip 
\noindent\textsc{Case 2:} $\hloc_{\mu}(T) = \infty$.
Now $\displaystyle \int_X g(x) d\mu = \infty$, where $g=\hloc_{\mu}$.
Let $M > 0$ be an arbitrary (large) number. Using Lemma~\ref{measurableh} and Lemma~\ref{L6.1} as above, there exists a  compact subset $K$ of $X$ such that all functions $g_m$ are continuous on $K$, $g_m$ converges uniformly to $g$ on $K$, and $$\int_{K}  g(y)\, d\mu \geq M .$$
Fix an integer $m_0 \geq 1$ so large that $g_m(x) \geq g(x) -1$ for all $x\in K$ and all $m \geq m_0$.
Define the sets $K_N$ as above with $\delta = 1$. Then again $(K_N)_{N\geq 1}$ is an increasing sequence of subsets of $K$ with $\bigcup_{N\geq 1} K_N = K$, so for a sufficiently large integer $N\geq1$,
$$\int_{K_N}  g(y)\, d\mu \geq \int_K g(y)\, d\mu - 1.$$
Now for $n \geq N$ and $m \geq m_0$, using estimates similar to those above, we get
\begin{align*}
M & \leq    \int_K g(y)\, d\mu  \leq \int_{K_N}  g(y)\, d\mu +1  \leq \int_{K_N}  g_m(y)\, d\mu  + 2=\\
& =     \int_{K_N}   \hloc_{\mu}(y,1/m) \, d\mu  + 2 \leq - \frac{1}{n} \int_{K_N}  \log \mu(D_n^\Gamma(y, 1/m)) \, d\mu (y) + 3 =\\
&=   - \frac{1}{n} \sum_{Y\in \aa^n_\Gamma} \int_{K_N \cap Y} \log \mu(D_n^\Gamma(y, 1/m)) \, d\mu (y) + 3.
\end{align*}
As before, if $Y \cap K_{N} \neq \e$ for some $Y \in \aa^n_\Gamma$, then $Y \subseteq D_n^\Gamma(y,1/m)$ and so $\mu(Y) \leq \mu(D_n^\Gamma(y,1/m)$
for any $y \in Y\cap K_{N}$. Thus, the above inequalities imply
\begin{align*}
M & \leq - \frac{1}{n} \sum_{Y\in \aa^n_\Gamma} \int_{K_N \cap Y} \log \mu(D_n^\Gamma(y, 1/m)) \, d\mu (y) + 3\\
&\leq  - \frac{1}{n} \sum_{Y\in \aa^n_\Gamma} \int_{K_N \cap Y} \log \mu(Y) \, d\mu + 3 \\
& \leq    - \frac{1}{n} \sum_{Y\in \aa^n_\Gamma} \mu(Y)\, \log \mu(Y)  + 3 .
\end{align*}
Letting $n \to \infty$ we obtain that $M \leq \th_{\mu}(T, \aa) + 3 \leq \th_{\mu}(T) + 3$, and  therefore $\th_\mu(T) \geq M-3$.
Letting $M \to \infty$ gives $\th_\mu(T) = \infty = \hloc_{\mu}(T)$.
\end{proof}

The following question arises in a natural way in view of the Brin-Katok Formula and Theorem~\ref{Proposition6.2}.

\begin{question}\label{Ques99}
Let $G$ be a semigroup, $\Gamma=(N_n)_{n\in\Z_+}$ a regular system in $G$, $T\colon G\times X\to X$ a continuous action of $G$ on a compact metric space $X$ and $\mu \in M(X,T)$. Does $\hloc_{\mu}(T)= \th_\mu(T)$ hold under suitable conditions?
\end{question}

Let $T\colon G\times X\to X$ be a continuous action of the semigroup $G$ on a compact metric space $X$, let $\Gamma=(N_n)_{n\in\Z_+}$ be a regular system in $G$ and $\mu$ a $T$-invariant Borel probability measure on $X$.
We denote by $\S_\mu(T,\Gamma)$ the \emph{essential supremum} of $\hloc_{\mu}$, that is,
$$\S_\mu(T,\Gamma)=\sup\{a\geq 0\colon\mu(\{ x \in X\colon \hloc_{\mu}(x,\Gamma) > a\})>0\}.$$
When there is no possibility of confusion we omit $\Gamma$ and write simply $\S_\mu(T)$.

\begin{remark}\label{7.7}
Let $T\colon G\times X\to X$ be a continuous action of the semigroup $G$ on a compact metric space $X$, and let $\Gamma=(N_n)_{n\in\Z_+}$ be a regular system in $G$. Let $\mu$ be a 
 Borel probability measure on $X$.
Put $s(Y):= \inf \{\hloc_{\mu}(x): x\in Y\}$ for  $Y \subseteq X$. Since $\S_\mu(T)> s(Y)$ for every measurable subset $Y$ of $X$ with $\mu(Y) >0$ and $\S_\mu(T)$ is the smallest number with this property, an alternative description of $\S_\mu(T)$ can be 
 $$ \S_\mu(T) = \sup\{s(Y)\colon Y\subseteq X\ \text{measurable},\ \mu(Y) >0\}. $$ 
\end{remark}

We conjecture that a local Variational Principle in the spirit of the one mentioned above by Zheng and Chen holds:

\begin{conjecture}\label{lVPconj}
If $G$ is a semigroup, $\Gamma=(N_n)_{n\in\Z_+}$ a regular system in $G$, $T\colon G\times X\to X$ a continuous action on a compact metric space $X$,  then
\begin{equation}\label{lVP}
\th(T,\Gamma)=\sup\{\underline h_\mu^{loc}(T,\Gamma)\colon \mu\ \text{Borel probability measure on $X$}\}. 
\end{equation}
\end{conjecture}

As mentioned above, the validity of the conjecture for $\Z_+$-actions was established by Feng and Huang \cite{FH}.
The following result can be seen as one half of the local Variational Principle \eqref{lVP}, namely, $\hloc_{\mu}(T) \leq \th(T)$ for every Borel probability measure $\mu$. Therefore, in case Question~\ref{Ques99}  has a positive answer, this would provide also one half of the Variational Principle \eqref{7.2}, namely, $ \th_\mu(T) \leq \th(T)$ for every $\mu\in M(X,T)$.

\begin{theorem}\label{Proposition6.1} 
Let $G$ be a semigroup, $\Gamma=(N_n)_{n\in\Z_+}$ a regular system in $G$ and $T\colon G\times X\to X$ a continuous action of $G$ on a compact metric space $X$. Then, for a Borel probability measure $\mu$ on $X$, $\S_{\mu}(T) \leq c(T),$
and therefore $\S_{\mu}(T) \leq \th(T)$. In particular,  $\hloc_{\mu}(T) \leq \th(T)$.
\end{theorem}
\begin{proof}
Set $S = \S_{\mu}(T)$ and take a small constant $\delta > 0$.  Then $\mu(Y_0)> 0,$ where $Y_0= \{ x \in X : \hloc_{\mu}(x) > S - \delta\}$. 
Since the measure is inner regular, there exists a compact subset $K$ of $Y_0$ such that $\mu(K)>0$ and 
$$\hloc_{\mu}(x) > S - \delta\ \text{for every}\ x \in K.$$

For every integer $m \geq 1$ set
$$K_m = \left\{ x\in K\colon \liminf_{n\to\infty} - \frac{1}{n} \log \mu(D_n^\Gamma(x, r)) > S - \delta, \: \forall \: r \in (0,1/m] \right\} .$$
The sequence $(K_m)_{m\geq1}$ is increasing and $\bigcup_{m\geq1}K_m = K$, so there exists $m_{0} \geq 1$ with
$\mu(K_{m_{0}}) > \frac{1}{2} \mu(K) > 0$. 

Similarly, for every $N\geq1$ let
$$K_{m_{0},N} = \left\{ x\in K_{m_0}\colon - \frac{1}{n} \log \mu(D_n^\Gamma(x, r))>S-\delta, \: \forall \: n \geq N , \: \forall \: r \in (0,1/m) \right\};$$
the sequence of sets $(K_{m_0,N})_{N\geq1}$ is increasing and $\bigcup_{N\geq1}K_{m_0,N}=K_{m_0}$, so there exists {$N_0 \geq 1$} such that
$$\mu(K_{m_0,N_0}) > \frac{1}{2} \mu(K_{m_0}) > \frac{1}{4} \mu(K) > 0.$$

Set $F=K_{m_0,N_0}$ and assume for a contradiction that $c(T,F) < S- \delta$. Fix an arbitrary $\lambda\in\R$ with
$c(T,F) < \lambda < S-\delta$. Then $\om_{\lambda}(F) = 0$ by the definition of $c(T,F)$. This implies that there exist $\epsilon\in (0,1/m_0)$, 
an integer $N \geq N_0$ and a finite or countable cover $\{D_{n_i}^\Gamma(x_i, \epsilon)\colon i\in I\}$ of $F$ with $x_i \in F$ and 
$n_i \geq N_0$ for all $i\in I$, such that
\begin{equation}\label{6.2}
\sum_{i\in I} e^{-\lambda n_i} < \frac{\mu(F)}{2} .
\end{equation}
Then by the choice of $F$, for all $i\in I$ we have $-  \log \mu(D_{n_i}^\Gamma(x_i, \epsilon)) > n_i (S - \delta)$, so
$\mu(D_{n_i}^\Gamma(x_i,\epsilon)) < e ^{-n_i (S - \delta)}.$
Since $F \subseteq \bigcup_{i\in I} D_{n_i}^\Gamma(x_i,\epsilon)$, it now follows that
$$\sum_{i\in I} e^{-\lambda n_i} > \sum_{i\in I} e^{-(S-\delta) n_i} > \sum_{i\in I} \mu(D_{n_i}^\Gamma(x_i,\epsilon))\geq \mu(F),$$
and this is a contradiction with \eqref{6.2}.

In this way we have proved that $c(T,F) \geq S-\delta$. Letting $\delta \to 0$ {we obtain that} $c(T,F) \geq S$ which in turn
implies $c(T) \geq S$. By Theorem~\ref{Theorem5.5} and Remark~\ref{Remark5.4}, we conclude also that $S \leq c(T) = b(T) \leq \th(T)$.

\smallskip 
The last assertion follows from the inequality  $\hloc_{\mu}(T) \leq \S_{\mu}(T)$ since by definition $\hloc_{\mu}(T)$ and $\S_\mu(T)$ are respectively the integral and the essential supremum of the function $\hloc_\mu\colon X\to \R_{\geq0}$. 
\end{proof}

\begin{remark} 
Theorem \ref{Proposition6.1} is closely related to \cite[Corollary 5.4]{Bis3}, where $G$ is a finitely generated group acting on a compact metric space $X$ and the regular system $\Gamma$ is standard. In the notation of Theorem \ref{Proposition6.1} and Remark~\ref{7.7}, it is proved in \cite[Corollary 5.4]{Bis3} that if $E$ is a Borel subset of $X$ with $\mu(E)>0$, then $s(E)\leq \th(T,E)$. 
Therefore, in view of Remark~\ref{7.7}, $\S_{\mu}(T) \leq  \th(T).$
\end{remark}

The following question arises as very natural in view of Theorem~\ref{Proposition6.1}. 

\begin{question} 
Let $G$ be a semigroup, $\Gamma=(N_n)_{n\in\Z_+}$ a regular system in $G$ and $T\colon G\times X\to X$ a continuous action of $G$ on a compact metric space $X$. Does $c(T)=\sup\{\underline S_{\mu}(T)\colon \mu\ \text{Borel probability measure on $X$}\}$ hold?
Does $\th(T)=\sup\{\underline S_{\mu}(T)\colon \mu\ \text{Borel probability measure on $X$}\}$ hold at least when $G$ is commutative and finitely generated and $\Gamma$ is standard?
\end{question}

In view of the Variational Principle and of Conjecture~\ref{lVPconj} it makes sense to ask whether also the following version of the local Variational Principle holds.

\begin{question}\label{laaaaaaast:question}
If $G$ is a semigroup, $\Gamma=(N_n)_{n\in\Z_+}$ a regular system in $G$, $T\colon G\times X\to X$ a continuous action on a compact metric space $X$ admitting some $T$-invariant Borel probability measure, does
\begin{equation}\label{lVP*}
\th(T,\Gamma)=\sup\{\underline h_\mu^{loc}(T,\Gamma)\colon \mu\in M(X,T)\}?
\end{equation}
\end{question}

In view of Theorem~\ref{Proposition6.2},  this version \eqref{lVP*}  of the local Variational Principle would imply the ``hard half'' (i.e., the inequality $\leq$) of the Variational Principle \eqref{7.2}, as \eqref{lVP*} and Theorem~\ref{Proposition6.2} would give $\th(T)\leq \sup \{ \th_\mu(T)\colon \mu \in M(X,T) \}$. So only the inequality $\geq$ in \eqref{7.2}  (sometimes referred to as ``easy half'') of the Variational Principle  would be missing, now stated as \eqref{geqVP} in Conjecture~\ref{conjVP1/2}.

\end{document}